\documentclass[5pt]{article}
\usepackage{amsmath}
\usepackage[latin1]{inputenc}
\usepackage{amsmath,amsthm,amssymb}
\usepackage{graphicx}
\usepackage[all,poly,knot]{xy}

\usepackage[all]{xy}

\usepackage{listings} 
\usepackage{tikz-cd} 
\usepackage{amssymb} 
\usepackage{booktabs}
\usepackage{mathtools}
\usepackage{color}

\renewcommand{\theequation}{\thesection.\arabic{equation}}

\newtheorem{defn}{Definition}[section]
\newtheorem{lem}{Lemma}[section]
\newtheorem{thm}{Theorem} [section]

\newtheorem{exmp}{Example} [section]
\newtheorem{coro}{Corollary}[section]
\newtheorem{rem}{Remark}[section]
\newtheorem{conj}{Conjecture}[section]
\newtheorem{prob}{Problem}[section]

\title{Cyclic Codes and Cyclically Covering Subspaces over Finite Fields}

\author{Yangcheng Li$^{1,}$\footnote{E-mail\,$:$ liyc@m.scnu.edu.cn.} \,\, Pingzhi Yuan$^{1,}$\footnote{Corresponding author. E-mail\,$:$ yuanpz@scnu.edu.cn. Supported by NSF of China No.12571003, 12501006; Basic and Applied Basic Research Foundation of Guangdong Province No. 2024A1515010589.}  \\
	{\small\it  $^{1}$School of Mathematical Sciences, South China Normal University,}\\
	{\small\it Guangzhou 510631, Guangdong, P. R. China} \\
}

\date{}
\begin{document}
\baselineskip15pt \maketitle
\renewcommand{\theequation}{\arabic{section}.\arabic{equation}}
\catcode`@=11 \@addtoreset{equation}{section} \catcode`@=12

\begin{abstract}
Let \(q\) be a power of a prime \(p\), and let \(n\) be a positive integer. A subspace \(U\subseteq \mathbb F_q^n\) is called cyclically covering if the union of all its cyclic shifts covers \(\mathbb F_q^n\), and \(h_q(n)\) denotes the maximum possible codimension of such a subspace. This paper studies cyclically covering subspaces via cyclic codes. We first prove that \(h_q(n)=0\) if and only if every nonzero cyclic code in \(\mathbb F_q^n\) contains a full-weight codeword. We also relate \(h_q(n)\) to the maximum weights of cyclic codes. In particular, when \(h_q(n)>0\), we obtain sharp bounds for the maximum weight of cyclic codes without full-weight codewords and provide explicit examples attaining these bounds. Moreover, we study the number of cyclic codes containing no full-weight codeword. We determine this number completely over \(\mathbb F_2\), and give lower bounds over \(\mathbb F_3\). From this, we prove that if \(q\ge 3\) is an odd prime and \(m\ge 4\) is an integer, then \(h_q\left(\frac{q^m+1}{2}\right)>0\).
\end{abstract}


\section{Introduction}
A finite field, denoted as $\mathbb{F}_{q}$, where $q$ is a prime power, is a field that contains a finite number of elements. For $n\in\mathbb{N}$, let $\{\boldsymbol{e}_0, \boldsymbol{e}_1, \dots, \boldsymbol{e}_{n-1}\}$ be the standard basis for $\mathbb{F}^n_q$, the indices of vectors in $\mathbb{F}^n_q$ will be taken modulo $n$ (in particular, we set $\boldsymbol{e}_n = \boldsymbol{e}_0$). Define the cyclic shift operator $\tau : \mathbb{F}^n_q\to \mathbb{F}^n_q$ by
\[\tau: \sum_{i=0}^{n-1}a_i\boldsymbol{e}_i \mapsto \sum_{i=0}^{n-1}a_i\boldsymbol{e}_{i+1}.\]
We say that a subspace $U \subset \mathbb{F}^n_q$ is cyclically covering if $\bigcup_{i=0}^{n-1}\tau^i(U) =\mathbb{F}^n_q$. For any $n\in\mathbb{N}$, let $h_q(n)$ denote the largest possible codimension of a cyclically covering subspace of $\mathbb{F}^n_q$.

The determination of \(h_q(n)\) is a difficult problem in general. Only a limited number of exact values are known, and most available results concern special families of integers \(n\) or general upper and lower bounds. Motivated by Isbell's conjecture \cite{Isbell56, Isbell60}, Cameron, Ellis, and Raynaud \cite{Cameron-Ellis-Raynaud} in 2019 initiated a systematic study of this invariant and obtained several fundamental properties of \(h_q(n)\). We recall their main results as follows.
\begin{lem}\cite{Cameron-Ellis-Raynaud} \label{lem1.1} Let \(q\) be a power of prime \(p\), and \(n,m,d,k \in \mathbb{N}\), then the following hold.

$(\mathrm{i})$ $h_2(n) \ge 2$ for all odd integers $n > 3$.

$(\mathrm{ii})$ $h_q(nm) \geq \max\{h_q(n), h_q(m)\}.$

$(\mathrm{iii})$ \(h_q(n) \leq \lfloor\log_q(n)\rfloor\).

$(\mathrm{iv})$ $h_q(q^d - 1) = d - 1 = \lfloor\log_q(q^d - 1)\rfloor.$

$(\mathrm{v})$ \(h_q(M/c) \geq kd + k - c\frac{q^k - 1}{q - 1}\), where \(M = (q - 1)\left(\sum_{r=0}^d q^{kr}\right)\), and \(M\) has a divisor \(c \in \mathbb{N}\) such that \(c < (q - 1)\frac{q^{kd} - q^{-kd}}{q^k - 1}\).

$(\mathrm{vi})$ \(h_q\left(\sum_{r=0}^d q^{kr}\right) = kd\), if \(\gcd(d + 1, q^k - 1) = 1\).

$(\mathrm{vii})$ $h_q(kp^d) = 0$ if \(k \mid q - 1\).

$(\mathrm{viii})$ \(h_2(n) = 0\) if and only if \(n = 2^d\) for some \(d \in \mathbb{N} \cup \{0\}\), and \(h_2(n) = 1\) if and only if \(n = 3\).
\end{lem}

In 1991, Cameron (see \cite{Cameron} Problem 190) posed the following problem in an equivalent form:
\begin{prob}\cite{Cameron}\label{prob1.1}
Does $h_2(n)\to\infty$ as $n\to\infty$ over the odd integers or is $h_2(n) =2$ for infinitely many odd $n$?
\end{prob}
Motivated by Problem \ref{prob1.1}, in 2019, Aaronson, Groenland and Johnston \cite{Aaronson-Groenland-Johnston} investigated the cyclically covering subspaces of \(\mathbb{F}_2^n\). Their main conclusions are summarized as follows.
\begin{lem}\cite{Aaronson-Groenland-Johnston} \label{lem1.2} Let \(q\) be a power of prime \(p\), and \(n,m,\ell,d \in \mathbb{N}\), then the following hold.

$(\mathrm{i})$ \(h_q(mn) \geq h_q(m) + h_q(n)\).

$(\mathrm{ii})$ \(h_q(pn) \leq ph_q(n)\).

$(\mathrm{iii})$ \(h_q(\ell p^d) = 0\) for any \(\ell < q\).

$(\mathrm{iv})$ \(h_2(t) = 2\) if \(t > 3\) is a prime for which \(2\) is a primitive root.

$(\mathrm{v})$ \(h_q(t) = 0\), where \(q\) is an odd prime, and \(t > q\) is a prime with \(q\) as a primitive root.
\end{lem}
Based on conclusions $(\mathrm{i})$ and $(\mathrm{iv})$ of Lemma \ref{lem1.2}, a positive answer to Problem \ref{prob1.1} can be given provided that Artin's conjecture holds true. Artin conjectured that $2$ is a primitive root modulo infinitely many primes. More generally, for any non-square positive integer \(n\), there are infinitely many primes \(p\) for which \(n\) is a primitive root modulo \(p\). Artin's conjecture follows from the Generalized Riemann Hypothesis (Hooley \cite{Hooley}). Though no \(n\) is known to satisfy it, Heath-Brown \cite{Heath-Brown} proved it holds for at least one of \(\{2, 3, 5\}\).

Furthermore, Aaronson, Groenland and Johnston \cite{Aaronson-Groenland-Johnston} posed several interesting and challenging problems, such as:
\begin{prob}\cite{Aaronson-Groenland-Johnston}\label{prob1.2}
For which \( n \in \mathbb{N} \) is \( h_q(n) = 0 \)?
\end{prob}

In 2024, Huang \cite{Huang} obtained a necessary and sufficient condition for \(h_q(n) = 0\) when \(\gcd(q, n) = 1\).
His main result shows that this problem can be reduced to computing the values of the trace function over finite fields.
He also derived the following conclusions.
\begin{lem}\cite{Huang}\label{lem1.3} Let \( p \) be an odd prime, and let \( q \) be a power of \( p \). For any non-negative integer \( d \), the following statements hold:

$(\mathrm{i})$ \( h_q\left(p^d(q + 1)\right) = 0 \).

$(\mathrm{ii})$ \( h_q\left(2p^d(q - 1)\right) = 0 \) if \(4 \mid q + 1 \).

$(\mathrm{iii})$ Let \( \ell \) be an odd prime number such that \( q \) is a primitive root modulo \( 2\ell \). If \( q \) is relatively prime to \( \ell - 1 \), then \( h_q\left(2p^d\ell\right) = 0 \).
\end{lem}

In 2025, Sun, Ma and Zeng \cite{Sun-Ma-Zeng} investigated the cyclically covering subspaces of the finite field \(\mathbb{F}_{q^n}\) and determined the value of \(h_2(n)\) for certain special values of \(n\). In particular, they showed that \(h_2(21) = 4\). Finally, they established several lower bounds for \(h_q(n)\) in the case where \(\gcd(q, n) = 1\).

In 2025, Li and Yuan \cite{Li-Yuan} proved that the problem of determining \(h_q(n) = 0\) can be reduced to the case where \(\gcd(q, n) = 1\). Specifically, they established the following result.
\begin{lem}\cite{Li-Yuan}\label{lem1.4} 
Let \( q \) be a power of a prime $p$, and let \( n \) be a positive integer satisfying \( \gcd(p, n) = 1 \). Then for any non-negative integer \( k \), we have \( h_q(np^k) = 0 \) if and only if \( h_q(n) = 0 \).
\end{lem}

In 2026, Li, Yuan, Li and Zeng \cite{Li-Yuan-Li-Zeng} obtained several sufficient conditions for \(h_q(n) = 0\), which generalize some existing results. Specifically, they showed that \( h_q(\ell^t) = 0 \) whenever \( q \) is a primitive root modulo \( \ell^t \). Moreover, they proved that if \( n \) is odd and \( h_q(n) = 0 \), then also \( h_q(2n) = 0 \). 

In 2026, Li and Yuan \cite{Li-Yuan1} developed a discrete Fourier transform approach to studying cyclically covering subspaces of \(\mathbb F_q^n\) under the condition \(\gcd(n,q)=1\). As a consequence, they obtained necessary and sufficient conditions for \(h_q(n)=0\), and proved that this condition is equivalent to the existence of full-weight codewords in certain cyclic codes over \(\mathbb F_q\). This approach led to several applications. In the case where \(q\) and \(n\) are primes with \(n>q\) and \(q\) is a primitive root modulo \(n\), it gives a unified explanation of the contrast between the binary and odd characteristic cases: one has \(h_2(n)\ge 2\), whereas \(h_q(n)=0\) for odd primes \(q\). They further proved that \(h_3(n)\ge 1\) for every prime \(n>3\) with odd \(\operatorname{ord}_n(3)\), and, assuming the Generalized Riemann Hypothesis, that for every prime \(q>3\) there exist infinitely many primes \(n>q\) for which \(q\) is not a primitive root modulo \(n\) but \(h_q(n)=0\).  Moreover, they gave algebraic interpretations of the inequalities
\[h_q(mn)\ge \max\{h_q(m),h_q(n)\}\quad\text{and}\quad h_q(mn)\ge h_q(m)+h_q(n),\]
proved the Galois descent inequality
\[h_{q^m}(n)\le h_q(n),\]
generalized a class of constructions attaining the upper bound \(\lfloor \log_q n\rfloor\), and established average lower bounds for \(h_q(n)\) under the Generalized Riemann Hypothesis.

We also mention a problem related to covering vector spaces over \(\mathbb{F}_q\). Luh \cite{Luh} proved that any vector space \(V\) over a finite field \(\mathbb{F}_q\) can be expressed as the union of \(|\mathbb{F}_q| + 1\) proper subspaces, and such a collection of subspaces is unique up to automorphisms of \(V\).

The purpose of this paper is to develop a cyclic-code approach to the study of cyclically covering subspaces. Our first main result is a cyclic-code criterion for the vanishing of \(h_q(n)\): we prove that \(h_q(n)=0\) if and only if every nonzero cyclic code of length \(n\) over \(\mathbb F_q\) contains a full-weight codeword. This criterion does not require the coprimality assumption \(\gcd(q,n)=1\). As an application, we give a new proof of the reduction result
\[h_q(np^k)=0 \quad \Longleftrightarrow \quad h_q(n)=0,\]
where \(\gcd(n,p)=1\) and \(k\ge 0\). We also introduce admissible subspaces and give an algebraic proof of the inequality \(h_q(pn)\le p h_q(n)\).

The paper is organized as follows. In Section 2, we reduce the condition \(h_q(n)=0\) to the existence of full-weight codewords in nonzero cyclic codes. In Section 3, we relate \(h_q(n)\) to the maximum weight of cyclic codes. In particular, when \(h_q(n)>0\), we obtain sharp bounds for the maximum weight of cyclic codes without full-weight codewords and give explicit examples showing that these bounds are attained. In Section 4, we study the family of cyclic codes containing no full-weight codeword. We determine their number completely over \(\mathbb F_2\), obtain lower bounds over \(\mathbb F_3\), and construct an infinite family of examples: for any odd prime \(q\ge 3\) and integer \(m\ge 4\), we have \(h_q\left(\frac{q^m+1}{2}\right)>0.\) In Section 5, we prove that \(h_q(np^k)=0\) if and only if \(h_q(n)=0,\) where \(k\ge 0\) and \(\gcd(n,p)=1\). We then introduce admissible subspaces and prove the inequality \(h_q(pn)\le p h_q(n)\).

\section{Reduction to Cyclic Codes}
Since we are mainly concerned with the case \(h_q(n)=0\), this paper focuses on subspaces of codimension $1$. Let
\[V_{\boldsymbol{\alpha}}=\{\boldsymbol{x}\in \mathbb F_q^n :(\boldsymbol{x},\boldsymbol{\alpha})=0\},\qquad \boldsymbol{\alpha}\in \mathbb F_q^n\setminus\{\boldsymbol 0\},\]
be an \((n-1)\)-dimensional subspace of \(\mathbb F_q^n\). Denote
\([n]=\{0,1,\ldots,n-1\}.\) Recall that \(V_{\boldsymbol{\alpha}}\) is a cyclically covering subspace if and only if
\[\bigcup_{i=0}^{n-1}\tau^i(V_{\boldsymbol{\alpha}})=\mathbb F_q^n.\]
Equivalently, for every vector \(\boldsymbol{x}\in \mathbb F_q^n\), there exists an index \(i\in [n]\) such that
\(\boldsymbol{x}\in \tau^i(V_{\boldsymbol{\alpha}}).\)
For each \(i\in[n]\), we have the following equivalence
\[\boldsymbol{x}\in \tau^i(V_{\boldsymbol{\alpha}}) \Longleftrightarrow \tau^{-i}\boldsymbol{x}\in V_{\boldsymbol{\alpha}} \Longleftrightarrow
(\tau^{-i}\boldsymbol{x},\boldsymbol{\alpha})=0 \Longleftrightarrow (\boldsymbol{x},\tau^i\boldsymbol{\alpha})=0.\]
Hence we obtain
\(\tau^i(V_{\boldsymbol{\alpha}})=V_{\tau^i\boldsymbol{\alpha}}.\)
Thus the cyclically covering condition for \(V_{\boldsymbol{\alpha}}\) can be rewritten as
\[\bigcup_{i=0}^{n-1}\tau^i (V_{\boldsymbol {\alpha}})=\bigcup_{i=0}^{n-1}V_{\tau^i\boldsymbol{\alpha}}=\mathbb F_q^n.\]
Equivalently, \(V_{\boldsymbol{\alpha}}\) is a cyclically covering subspace if and only if
\[\forall\,\boldsymbol{x}\in\mathbb F_q^n,\quad \exists\, i\in[n]\quad\text{such that}\quad (\boldsymbol{x},\tau^i\boldsymbol{\alpha})=0.\]
Correspondingly, \(V_{\boldsymbol{\alpha}}\) fails to be cyclically covering if and only if there exists a vector \(\boldsymbol{x}\in\mathbb F_q^n\) such that
\[(\boldsymbol{x},\tau^i\boldsymbol{\alpha})\neq 0,\qquad \forall\, i\in[n].\]

It is worth mentioning that Aaronson, Groenland and Johnston \cite{Aaronson-Groenland-Johnston} introduced the notion of vectors that {\it work together} in 2019. Precisely, a vector \(\boldsymbol{\alpha} \in \mathbb{F}_q^n\) is said to {\it work} if for every vector \(\boldsymbol{x} \in \mathbb{F}_q^n\), there exists an integer \(i\in [n]\) such that the inner product \[(\boldsymbol{\alpha}, \tau^i (\boldsymbol{x}))=(\tau^{-i}(\boldsymbol{\alpha}), \boldsymbol{x}) = 0.\]
This condition is equivalent to \(V_{\boldsymbol{\alpha}}\) being a cyclically covering subspace.

A collection of vectors \(\boldsymbol{\alpha}_1,\boldsymbol{\alpha}_2,\dots,\boldsymbol{\alpha}_k\) is said to {\it work together} if for every vector \(\boldsymbol{x} \in \mathbb{F}_q^n\), there exists an integer \(i \in [n]\) such that
\[(\boldsymbol{\alpha}_1 , \tau^i (\boldsymbol{x})) = (\boldsymbol{\alpha}_2 , \tau^i (\boldsymbol{x})) = \dots = (\boldsymbol{\alpha}_k , \tau^i (\boldsymbol{x})) = 0.\]
Similarly, this characterization is equivalent to the intersection \(\bigcap_{j=1}^k V_{\boldsymbol{\alpha}_j}\) being a cyclically covering subspace.

A $k$-dimensional linear subspace \(\mathcal{C}\) of \(\mathbb{F}_q^n\) is called an \([n,k]_q\) linear code. A codeword \(\boldsymbol{c}\in\mathcal{C}\) is said to be a full-weight codeword if \(\operatorname{wt}(\boldsymbol{c})=n\).
In 2026, Li and Yuan \cite{Li-Yuan1} proved that under the condition \(\gcd(q,n)=1\), \(h_q(n)=0\) is equivalent to the statement that every nonzero cyclic code in \(\mathbb{F}_q^n\) contains a full-weight codeword. In fact, the assumption \(\gcd(q,n)=1\) can be removed, and we obtain the following result. The idea of our proof originates from the work of Aaronson, Groenland and Johnston \cite{Aaronson-Groenland-Johnston}.

For a fixed vector \(\boldsymbol{\alpha}\in\mathbb{F}_q^n\), define the map
\[A_{\boldsymbol{\alpha}}: \mathbb{F}_q^n \to \mathbb{F}_q^n,\quad
\boldsymbol{x} \mapsto A_{\boldsymbol{\alpha}}(\boldsymbol{x})
= \big((\boldsymbol{x},\boldsymbol{\alpha}),\, (\boldsymbol{x},\tau\boldsymbol{\alpha}),\, \dots,\, (\boldsymbol{x},\tau^{n-1}\boldsymbol{\alpha})\big).\]
Let \(\operatorname{Im} A_{\boldsymbol{\alpha}}\) denote the image of the map \(A_{\boldsymbol{\alpha}}\). We present the following theorem.
\begin{thm}\label{lem1}
The equality \(h_q(n)=0\) holds if and only if every nonzero cyclic code in \(\mathbb{F}_q^n\) contains a full-weight codeword.
\end{thm}

\begin{proof}
We first prove that the set \(\operatorname{Im} A_{\boldsymbol{\alpha}}\) is a cyclic code in \(\mathbb{F}_q^n\). It is trivial that \(A_{\boldsymbol{\alpha}}\) is a linear map, so \(\operatorname{Im} A_{\boldsymbol{\alpha}}\) is a linear subspace of \(\mathbb{F}_q^n\). The fact that \(\operatorname{Im} A_{\boldsymbol{\alpha}}\) is cyclic follows immediately from the identity
\[\tau\big(A_{\boldsymbol{\alpha}}(\boldsymbol{x})\big) = A_{\boldsymbol{\alpha}}\big(\tau(\boldsymbol{x})\big).\]
From the identity \((\tau^i \boldsymbol{\alpha})_j = \alpha_{j-i}\), we obtain \(\big(A_{\boldsymbol{\alpha}}(\boldsymbol{x})\big)_i = (\boldsymbol{x},\tau^i \boldsymbol{\alpha}) = \sum_{j=0}^{n-1} x_j \alpha_{j-i},\) where all subscripts are taken modulo \(n\).

Next we determine the generator polynomial of \(\operatorname{Im} A_{\boldsymbol{\alpha}}\). Let \(R_n = \mathbb{F}_q[X]/(X^n - 1)\). Define the map
\[\theta : \mathbb{F}_q^n \to R_n,\quad \boldsymbol{x} = (x_0,\dots,x_{n-1}) \mapsto \boldsymbol{x}(X) = \sum_{i=0}^{n-1} x_i X^i.\]
According to Proposition 3.1 in \cite{Aaronson-Groenland-Johnston}, for any vector  
\(\boldsymbol{\alpha}=(\alpha_0,\dots,\alpha_{n-1})\in\mathbb{F}_q^n\),  
we define its reversed vector \(\boldsymbol{\alpha}^*\) by  
\(\alpha_i^* = \alpha_{-i}\), with indices taken modulo \(n\).
It then follows that
\[\boldsymbol{\alpha}^*(X) = \sum_{i=0}^{n-1} \alpha_i^* X^i = \sum_{i=0}^{n-1} \alpha_{-i} X^i=\boldsymbol{\alpha}(X^{-1}) \pmod{X^n - 1},\]
or equivalently, \(\boldsymbol{\alpha}^*(X) = \alpha_0 + \alpha_{n-1} X + \alpha_{n-2} X^2 + \dots + \alpha_1 X^{n-1}.\)

We compute \(\boldsymbol{x}(X)\boldsymbol{\alpha}^*(X)\) in \(R_n\). Suppose
\[\boldsymbol{x}(X)\boldsymbol{\alpha}^*(X) \equiv c_0 + c_1X + \dots + c_{n-1}X^{n-1} \pmod{X^n - 1}.\]
Then
\(c_i = \sum_{j=0}^{n-1} x_j \alpha_{i-j}^*.\) Since
\(\alpha_{i-j}^* = \alpha_{-(i-j)} = \alpha_{j-i},\)
we have
\(c_i = \sum_{j=0}^{n-1} x_j \alpha_{j-i} = \big(\boldsymbol{x}, \tau^i \boldsymbol{\alpha}\big).\) Hence
\[\boldsymbol{x}(X)\boldsymbol{\alpha}^*(X) = \sum_{i=0}^{n-1} \big(\boldsymbol{x}, \tau^i \boldsymbol{\alpha}\big) X^i.\]
Equivalently,
\(\boldsymbol{x}(X)\boldsymbol{\alpha}^*(X) = \theta\big(A_{\boldsymbol{\alpha}}(\boldsymbol{x})\big).\)
It follows that
\[\theta\big(\operatorname{Im} A_{\boldsymbol{\alpha}}\big) = \big\{\theta\big(A_{\boldsymbol{\alpha}}(\boldsymbol{x})\big): \boldsymbol{x} \in \mathbb{F}_q^n\big\} = \big\langle \boldsymbol{\alpha}^{*}(X) \big\rangle \subseteq R_n.\]
Let \(g_{\boldsymbol \alpha}(X)=\gcd(\boldsymbol{\alpha}^{*}(X),X^n-1)\). Then \(g_{\boldsymbol \alpha}(X)\) is exactly the generator polynomial of \(\operatorname{Im} A_{\boldsymbol{\alpha}}\). By the arbitrariness of \(\boldsymbol{\alpha}^*(X)\), we conclude that \(\operatorname{Im} A_{\boldsymbol{\alpha}}\) can represent any cyclic code in \(\mathbb{F}_q^n\).

Then \(V_{\boldsymbol{\alpha}}\) is cyclically covering if and only if for every \(\boldsymbol{x} \in \mathbb{F}_q^n\), there exists some integer \(i\) such that
\((\boldsymbol{x},\tau^i \boldsymbol{\alpha}) = 0.\)
This condition is equivalent to \(\mathrm{wt}\big(A_{\boldsymbol{\alpha}}(\boldsymbol{x})\big) < n\) for all $\boldsymbol{x} \in \mathbb{F}_q^n$.
Conversely, \(V_{\boldsymbol{\alpha}}\) is not cyclically covering if and only if there exists some \(\boldsymbol{x} \in \mathbb{F}_q^n\) satisfying
\(\mathrm{wt}\big(A_{\boldsymbol{\alpha}}(\boldsymbol{x})\big) = n.\)
Therefore, \(h_q(n)=0\) holds if and only if every nonzero cyclic code in \(\mathbb{F}_q^n\) contains full-weight codewords.
\end{proof}

In Theorem \ref{lem1}, we established a correspondence between vectors \(\boldsymbol{\alpha}\in\mathbb{F}_q^n\) and the ideals \(\langle g_{\boldsymbol{\alpha}}(X)\rangle\) in \(R_n\). However, this correspondence is not bijective. In fact, we have the following theorem. 

Let \(q\) be a power of a prime \(p\), and let \(N=np^k\), where \(\gcd(n,p)=1\) and \(k\ge 0\). Suppose
\[X^N - 1 = \prod_{i=1}^{r} f_i(X)^{p^k},\]
where the polynomials \(f_i(X)\) are distinct and irreducible. Given a vector \(\boldsymbol{\alpha} \in \mathbb{F}_q^N\), let the ideal associated to \(\boldsymbol{\alpha}\) be \(\langle g_{\boldsymbol{\alpha}}(X) \rangle\), and write
\[g_{\boldsymbol{\alpha}}(X) = \prod_{i=1}^{r} f_i(X)^{s_i},\qquad 0 \le s_i \le p^k.\]
Also define
\[h_{\boldsymbol{\alpha}}(X) = \frac{X^N - 1}{g_{\boldsymbol{\alpha}}(X)} = \prod_{i=1}^{r} f_i(X)^{p^k - s_i}.\]

\begin{thm}\label{thm2.5}
The number of vectors \(\boldsymbol{\beta} \in \mathbb{F}_q^N\) satisfying \(\langle g_{\boldsymbol{\alpha}}(X) \rangle = \langle g_{\boldsymbol{\beta}}(X) \rangle\) is
\[\#\big\{\boldsymbol{\beta} \in \mathbb{F}_q^N : \langle g_{\boldsymbol{\alpha}}(X) \rangle = \langle g_{\boldsymbol{\beta}}(X) \rangle\big\} = q^{\deg h_{\boldsymbol{\alpha}}} \prod_{f_i \mid h_{\boldsymbol{\alpha}}} \left(1 - q^{-\deg f_i}\right).\]
\end{thm}

\begin{proof}
Let $R_N=\mathbb F_q[X]/(X^N-1).$
For a vector \(\boldsymbol{\beta}\in \mathbb F_q^N\), write \(\boldsymbol \beta^*(X)=\boldsymbol \beta(X^{-1}) \pmod{X^N-1}.\)
The map
\(\boldsymbol{\beta}\mapsto \boldsymbol \beta^*(X)\)
is a bijection from \(\mathbb F_q^N\) onto \(R_N\). Hence it suffices to count the number of elements \(a(X)\in R_N\) such that
\(\langle a(X)\rangle=\langle g_{\boldsymbol{\alpha}}(X)\rangle.\) Let
\[g_{\boldsymbol{\alpha}}(X)=\gcd(\boldsymbol \alpha^*(X),X^N-1),\qquad
h_{\boldsymbol{\alpha}}(X)=\frac{X^N-1}{g_{\boldsymbol{\alpha}}(X)}.\]
Then \(X^N-1=g_{\boldsymbol{\alpha}}(X)h_{\boldsymbol{\alpha}}(X).\)
Every element of the ideal \(\langle g_{\boldsymbol{\alpha}}(X)\rangle\subseteq R_N\) can be written in the form
\(g_{\boldsymbol{\alpha}}(X)u(X)\pmod{X^N-1}.\)

We claim that \(\langle g_{\boldsymbol{\alpha}}(X)u(X)\rangle=\langle g_{\boldsymbol{\alpha}}(X)\rangle\)
if and only if \(u(X)\) is a unit modulo \(h_{\boldsymbol{\alpha}}(X)\), i.e.
\[u(X)\in \left(\mathbb F_q[X]/(h_{\boldsymbol{\alpha}}(X))\right)^*.\]
Indeed, if \(\langle g_{\boldsymbol{\alpha}}(X)u(X)\rangle=\langle g_{\boldsymbol{\alpha}}(X)\rangle\), then there exists a polynomial \(v(X)\) such that
\[g_{\boldsymbol{\alpha}}(X)\equiv v(X)g_{\boldsymbol{\alpha}}(X)u(X)\pmod{X^N-1}.\]
Since \(X^N-1=g_{\boldsymbol{\alpha}}(X)h_{\boldsymbol{\alpha}}(X)\), this is equivalent to
\[g_{\boldsymbol{\alpha}}(X)\bigl(1-v(X)u(X)\bigr)\equiv 0\pmod{g_{\boldsymbol{\alpha}}(X)h_{\boldsymbol{\alpha}}(X)}.\]
Hence \(h_{\boldsymbol{\alpha}}(X)\mid 1-v(X)u(X),\) which means that \(u(X)\) is invertible modulo \(h_{\boldsymbol{\alpha}}(X)\).

Conversely, if \(u(X)\) is invertible modulo \(h_{\boldsymbol{\alpha}}(X)\), then there exists \(v(X)\in \mathbb F_q[X]\) such that
\(u(X)v(X)\equiv 1\pmod{h_{\boldsymbol{\alpha}}(X)}.\)
Thus
\(u(X)v(X)-1=h_{\boldsymbol{\alpha}}(X)w(X)\)
for some \(w(X)\in\mathbb F_q[X]\). Multiplying by \(g_{\boldsymbol{\alpha}}(X)\), we get
\[g_{\boldsymbol{\alpha}}(X)u(X)v(X)-g_{\boldsymbol{\alpha}}(X)
= g_{\boldsymbol{\alpha}}(X)h_{\boldsymbol{\alpha}}(X)w(X)
\equiv 0\pmod{X^N-1}.\]
Therefore \(g_{\boldsymbol{\alpha}}(X)\in \langle g_{\boldsymbol{\alpha}}(X)u(X)\rangle,\) and hence
\(\langle g_{\boldsymbol{\alpha}}(X)u(X)\rangle=\langle g_{\boldsymbol{\alpha}}(X)\rangle.\)
This proves the claim.

Consequently, the number of elements \(a(X)\in R_N\) satisfying
\(\langle a(X)\rangle=\langle g_{\boldsymbol{\alpha}}(X)\rangle\)
is precisely
\[\left|\left(\mathbb F_q[X]/(h_{\boldsymbol{\alpha}}(X))\right)^*\right|.\]
It remains to compute the number of units in \(\mathbb F_q[X]/(h_{\boldsymbol{\alpha}}(X))\). Since
\[h_{\boldsymbol{\alpha}}(X)=\frac{X^N-1}{g_{\boldsymbol{\alpha}}(X)}
=\prod_{i=1}^r f_i(X)^{p^k-s_i},\]
the Chinese remainder theorem gives
\[\mathbb F_q[X]/(h_{\boldsymbol{\alpha}}(X)) \cong \prod_{p^k-s_i>0}\mathbb F_q[X]/\left(f_i(X)^{p^k-s_i}\right).\]
For an irreducible polynomial \(f_i(X)\) of degree \(d_i=\deg f_i\), the local ring \(\mathbb F_q[X]/\left(f_i(X)^a\right)\)
has \(q^{ad_i}\) elements, and its non-units are exactly the elements divisible by \(f_i(X)\). Therefore the number of units in it is
\[q^{ad_i}-q^{(a-1)d_i}=q^{ad_i}\left(1-q^{-d_i}\right).\]
Taking the product over all indices with \(p^k-s_i>0\), we obtain
\[\left|\left(\mathbb F_q[X]/(h_{\boldsymbol{\alpha}}(X))\right)^*\right|=\prod_{p^k-s_i>0}
q^{(p^k-s_i)\deg f_i}\left(1-q^{-\deg f_i}\right).\]
Since
\[\deg h_{\boldsymbol{\alpha}}=\sum_{p^k-s_i>0}(p^k-s_i)\deg f_i,\]
this becomes
\[\left|\left(\mathbb F_q[X]/(h_{\boldsymbol{\alpha}}(X))\right)^*\right|=q^{\deg h_{\boldsymbol{\alpha}}}
\prod_{f_i\mid h_{\boldsymbol{\alpha}}}\left(1-q^{-\deg f_i}\right).\]
Therefore
\[\#\big\{\boldsymbol{\beta} \in \mathbb{F}_q^N : \langle g_{\boldsymbol{\alpha}}(X) \rangle=\langle g_{\boldsymbol{\beta}}(X) \rangle\big\}=q^{\deg h_{\boldsymbol{\alpha}}}
\prod_{f_i\mid h_{\boldsymbol{\alpha}}}\left(1-q^{-\deg f_i}\right).\]
\end{proof}

\begin{rem}\label{rem2.1}
If \(\gcd(n,q)=1\), Theorem \ref{thm2.5} admits some simplifications. In the coprime case, the polynomial \(X^n-1\) is square-free. This corresponds to the case \(k=0\) in Theorem~\ref{thm2.5}. Hence \(X^n-1=\prod_{i=1}^{r} f_i(X).\)
For a vector \(\boldsymbol{\alpha}\in\mathbb F_q^n\), write \(g_{\boldsymbol{\alpha}}(X)=\prod_{i\in J} f_i(X)\) for some subset \(J\subseteq\{1,\ldots,r\}\). Then
\[h_{\boldsymbol{\alpha}}(X)=\frac{X^n-1}{g_{\boldsymbol{\alpha}}(X)}=\prod_{i\notin J} f_i(X).\]
Therefore Theorem~\ref{thm2.5} reduces to
\[\#\bigl\{\boldsymbol{\beta}\in\mathbb F_q^n:\langle g_{\boldsymbol{\beta}}(X)\rangle=\langle g_{\boldsymbol{\alpha}}(X)\rangle\bigr\}=\prod_{i\notin J}\left(q^{\deg f_i}-1\right).\]
\end{rem}

\begin{coro}\label{coro2.2}Let \(\mathcal C=\langle g(X)\rangle\) be a nonzero cyclic code containing no full-weight codeword in \(R_N=\mathbb F_q[X]/(X^N-1)\), and write the factorization
\[X^N-1=\prod_{i=1}^r f_i(X)^{p^k},\qquad g(X)=\prod_{i=1}^r f_i(X)^{s_i},\]
where the integers satisfy \(0\le s_i\le p^k\). Set
\(h(X)=(X^N-1)/g(X).\)
Then the total number of \((N-1)\)-dimensional cyclically covering subspaces \(V_{\boldsymbol\beta}\subseteq\mathbb F_q^N\) satisfying the image condition \(\theta(\operatorname{Im}A_{\boldsymbol\beta})=C\) is
\[\frac{1}{q-1}\,q^{\deg h}\prod_{f_i\mid h}\left(1-q^{-\deg f_i}\right).\]
\end{coro}

\begin{proof}
This corollary follows immediately from the following fact. For any nonzero vectors \(\boldsymbol\beta,\boldsymbol\gamma\in\mathbb F_q^N\), we have the equivalence
\[V_{\boldsymbol\beta}=V_{\boldsymbol\gamma}\iff\exists\,\lambda\in\mathbb F_q^\ast,\ \boldsymbol\gamma=\lambda\boldsymbol\beta.\]
\end{proof}

\section{Maximum Weight of Cyclic Codes}
Nearly all research concerning the weights of cyclic codes focuses on minimum weights, while studies on maximum weights remain quite scarce. In Theorem \ref{lem1}, we prove that \(h_q(n)=0\) is equivalent to every nonzero cyclic code in \(\mathbb{F}_q^n\) containing a full-weight codeword. This establishes a connection between the maximum weight of cyclic codes and the value of \(h_q(n)\). 

The standard parameters of a cyclic code $\mathcal C$ are written as \([n,k,d(\mathcal C)]_q\), where $k$ denotes the dimension of $\mathcal C$ and \(d(\mathcal C)\) stands for the minimum Hamming weight of codewords in $\mathcal C$ (See \cite{Huffman-Pless} for more details). We now additionally take the maximum weight of codewords in $\mathcal C$ into consideration and denote it by \(D(\mathcal C)\). The parameters of $\mathcal C$ are thus written as \([n,k,d(\mathcal C),D(\mathcal C)]_q\), with
\[d(\mathcal C)=\min\{\operatorname{wt}(\boldsymbol c):\boldsymbol 0 \neq \boldsymbol c\in \mathcal C\},\quad D(\mathcal C)=\max\{\operatorname{wt}(\boldsymbol c): \boldsymbol c\in \mathcal C\}.\]
We present the following theorem.

\begin{thm}\label{thm3.1}
If $h_q(n)=0$, then $D(\mathcal C)=n$ holds for every non-zero cyclic code $\mathcal C\subset \mathbb{F}_q^n$. If \(h_q(n)>0\), then there exists a nonzero cyclic code \(\mathcal C\subseteq \mathbb F_q^n\) with \(D(\mathcal C)<n\). Moreover, for every nonzero cyclic code \(\mathcal C\subseteq \mathbb F_q^n\) of dimension \(k\) satisfying \(D(\mathcal C)<n\), one has \(k\ge 2\) and
\[\left\lceil \frac{n(q-1)q^{k-1}}{q^k-1}\right\rceil\le D(\mathcal C)\le n-1.\]
\end{thm}

\begin{proof}
The case where $h_q(n)=0$ corresponds to Theorem \ref{lem1}. If \(h_q(n)>0\), there exists at least one nonzero cyclic code \(\mathcal C\subset\mathbb{F}_q^n\) satisfying \(D(\mathcal C)\leq n-1\). We next prove that the lower bound of \(D(\mathcal C)\) is \(\left\lceil \frac{n(q-1)q^{k-1}}{q^k-1}\right\rceil\). 

For any \(i\in[n]\), define the projection map \(\pi_i\colon \mathcal C\to\mathbb{F}_q\) by \(\pi_i(\boldsymbol c)=c_i\). We claim that for any nonzero cyclic code $\mathcal C$ and arbitrary coordinate \(i\in[n]\), there exists at least one codeword \(\boldsymbol c\in \mathcal C\) such that \(\pi_i(\boldsymbol c)=c_i\neq0\); otherwise, the definition of cyclic codes forces \(\mathcal C=\{\boldsymbol 0\}\). Therefore, \(\pi_i\) is surjective for every \(i\in[n]\), which implies \(\dim\ker\pi_i=k-1\). Hence, the number of codewords with a nonzero entry at the $i$-th coordinate is
\[q^k - q^{k-1} = q^{k-1}(q-1).\]
By symmetry of cyclic codes, this quantity holds for all $n$ coordinates. We thus obtain
\[\sum_{\boldsymbol c\in \mathcal C} \operatorname{wt}(\boldsymbol c) = n \cdot q^{k-1}(q-1).\]
The average weight of nonzero codewords is then
\[\overline{\operatorname{wt}} = \frac{\displaystyle\sum_{\boldsymbol c\neq \boldsymbol 0} \operatorname{wt}(\boldsymbol c)}{q^k-1} = \frac{n\, q^{k-1}(q-1)}{q^k-1}.\]
The maximum weight cannot be smaller than the average weight. Since all Hamming weights are integers, we conclude
\[D(\mathcal C) \ge \left\lceil \frac{n(q-1) q^{k-1}}{q^k - 1} \right\rceil.\]
Finally, we show that
$k \geq 2.$
When $k=1$, the preceding lower bound reduces to
$D(\mathcal C) \geq n,$
which contradicts $D(\mathcal C) \leq n-1$.
\end{proof}

In the following, we use the DFT method from \cite{Li-Yuan1} to explicitly construct cyclic codes whose maximum weights attain the upper and lower bounds in Theorem \ref{thm3.1}. To this end, we need to briefly introduce the basic concepts and properties of the DFT.

Assume \(\gcd(n,q)=1\). Let \(m=\operatorname{ord}_n(q)\). There exists a primitive $n$-th root of unity $\omega$ in the extension field $\mathbb{F}_{q^m}$ over $\mathbb{F}_q$ such that $\omega^n=1$ and $\omega^i\neq1$ for $1\le i\le n-1$. For a vector
\(\boldsymbol{x}=(x_0,x_1,\ldots,x_{n-1})\in \mathbb F_q^n,\)
its discrete Fourier transform (DFT) is the vector
\(\widehat{\boldsymbol{x}}=(\widehat{\boldsymbol{x}}_0,\widehat{\boldsymbol{x}}_1,\ldots,\widehat{\boldsymbol{x}}_{n-1})\in \mathbb F_{q^m}^n,\)
where
\[\widehat{\boldsymbol{x}}_k=\operatorname{DFT}(\boldsymbol{x})_k=\sum_{j=0}^{n-1}x_j\omega^{jk},\qquad k\in [n].\]
Since \(\gcd(n,q)=1\), the element \(n\) is invertible in \(\mathbb F_q\). The inverse discrete Fourier transform is given by
\[x_j=\operatorname{IDFT}(\widehat{\boldsymbol{x}})_j=\frac1n\sum_{k=0}^{n-1}\widehat{\boldsymbol{x}}_k\omega^{-jk},\qquad j\in [n],\]
where \(\frac 1n\) denotes the inverse of \(n\) in \(\mathbb F_q\).

We record several standard facts about the DFT that will be used later.

\begin{lem}\cite{Li-Yuan1}\label{lem3-1}
The image of \(\mathbb F_q^n\) under the DFT is 
\[\operatorname{DFT}(\mathbb F_q^n)=\left\{\boldsymbol{y}\in \mathbb F_{q^m}^n : y_k^q=y_{qk}\ \text{for all }k\in [n] \right\}.\]
\end{lem}

\begin{lem}\cite{Li-Yuan1}\label{lem3-2}
For every \(\boldsymbol{x}\in\mathbb F_q^n\) and every \(k\in[n]\), we have \(\widehat{\tau(\boldsymbol{x})}_k=\omega^k\widehat{\boldsymbol{x}}_k.\)
\end{lem}

\begin{lem}\cite{Li-Yuan1}\label{lem3-3}
For any \(\boldsymbol{x}=(x_0,\ldots,x_{n-1}),\boldsymbol{y}=(y_0,\ldots,y_{n-1})\) in \(\mathbb F_q^n\), let the standard inner product be \((\boldsymbol{x},\boldsymbol{y})=\sum_{i=0}^{n-1}x_i y_i.\)
Then
\[(\boldsymbol{x},\boldsymbol{y})=\frac1n\sum_{k=0}^{n-1}\widehat{\boldsymbol{x}}_k\,\widehat{\boldsymbol{y}}_{-k}.\]
\end{lem}

For \(j\in[n]\), the \(q\)-cyclotomic coset of \(j\) modulo \(n\) is defined by
\[C_j=\{j,jq,jq^2,\ldots,jq^{m_j-1}\}\pmod n,\]
where \(m_j\) is the smallest positive integer such that \(jq^{m_j}\equiv j\pmod n.\)

Since \(\gcd(n,q)=1\), the polynomial \(X^n-1\) is square-free over \(\mathbb F_q\). Write its irreducible factorization over \(\mathbb F_q[X]\) as
\[X^n-1=f_1(X)f_2(X)\cdots f_r(X).\]
Each irreducible factor \(f_s(X)\) corresponds to a unique \(q\)-cyclotomic coset. More precisely, if \(f_s(\omega^j)=0\), then the roots of \(f_s(X)\) are exactly \(\{\omega^k:k\in C_j\}.\) We denote this cyclotomic coset by \(C(s)\).

\begin{lem}\cite{Li-Yuan1}\label{lem3-4}
We have the direct sum decomposition
\[\mathbb F_q^n=W_1\oplus W_2\oplus\cdots\oplus W_r,\]
where
\(W_i=\ker f_i(\tau)=\{\boldsymbol{x}\in\mathbb F_q^n :	f_i(\tau)(\boldsymbol{x})=\boldsymbol{0}\}.\)
\end{lem}

If \(f_i(\tau)(\boldsymbol{x})=\boldsymbol{0}\), then \(f_i(X)\) is an annihilating polynomial of \(\boldsymbol{x}\), and \(\boldsymbol{x}\in W_i\). For \(\boldsymbol{x}\in\mathbb F_q^n\), we define the support of its DFT by
\[\operatorname{supp}(\widehat{\boldsymbol{x}})=\{k\in[n]:\widehat{\boldsymbol{x}}_k\ne 0\}.\]
The DFT support will play an important role in the following arguments. We shall use the following two elementary properties.

\begin{lem}\cite{Li-Yuan1}\label{lem3-5}
For every \(\boldsymbol{x}\in\mathbb F_q^n\), the set \(\operatorname{supp}(\widehat{\boldsymbol{x}})\) is a union of \(q\)-cyclotomic cosets modulo \(n\).
\end{lem}

\begin{lem}\cite{Li-Yuan1}\label{lem3-6}
For \(1\le i\le r\), if \(\boldsymbol{x}\in W_i\) and \(\boldsymbol{x}\ne\boldsymbol{0}\), then
\(\operatorname{supp}(\widehat{\boldsymbol{x}})=C(i).\)
More generally, suppose
\(\boldsymbol{x}\in W_{i_1}\oplus W_{i_2}\oplus\cdots\oplus W_{i_s},\)
and write
\(\boldsymbol{x}=\boldsymbol{\alpha}_1+\boldsymbol{\alpha}_2+\cdots+\boldsymbol{\alpha}_s,\boldsymbol{\alpha}_t\in W_{i_t}, \boldsymbol{\alpha}_t\ne \boldsymbol{0}.\)
Then
\[\operatorname{supp}(\widehat{\boldsymbol{x}})=\bigcup_{t=1}^s C(i_t).\]
\end{lem}

In addition, we include a lemma that will be needed later.
\begin{lem}\label{lem3-7}
Assume \(\gcd(n,q)=1\). For $\boldsymbol{x}=(x_0,\ldots,x_{n-1}), \boldsymbol{y}=(y_0,\ldots,y_{n-1})\in \mathbb F_q^n,$ define their coordinatewise product by
\[\boldsymbol{x}\odot\boldsymbol{y}=(x_0y_0,x_1y_1,\ldots,x_{n-1}y_{n-1}).\]
Then for every \(\ell\in[n]\), we have
\[\widehat{\boldsymbol{x}\odot\boldsymbol{y}}_{\ell}=\frac1n \sum_{\substack{a,b\in[n]\\ a+b\equiv \ell \pmod n}}\widehat{\boldsymbol{x}}_{a}\,\widehat{\boldsymbol{y}}_{b}.\]
\end{lem}

\begin{proof}
By the inverse DFT formula, for every \(j\in[n]\) we have
\[x_j=\frac1n\sum_{a=0}^{n-1}\widehat{\boldsymbol{x}}_a\omega^{-ja}, \qquad y_j=\frac1n\sum_{b=0}^{n-1}\widehat{\boldsymbol{y}}_b\omega^{-jb}.\]
Therefore,
\[\begin{aligned}
\widehat{\boldsymbol{x}\odot\boldsymbol{y}}_{\ell}
&=\sum_{j=0}^{n-1}x_jy_j\omega^{j\ell} =\sum_{j=0}^{n-1}
\left(\frac1n\sum_{a=0}^{n-1}\widehat{\boldsymbol{x}}_a\omega^{-ja}\right)
\left(\frac1n\sum_{b=0}^{n-1}\widehat{\boldsymbol{y}}_b\omega^{-jb}\right)
\omega^{j\ell} =\frac1{n^2}
\sum_{a=0}^{n-1}\sum_{b=0}^{n-1}
\widehat{\boldsymbol{x}}_a\widehat{\boldsymbol{y}}_b
\sum_{j=0}^{n-1}\omega^{j(\ell-a-b)}.
\end{aligned}\]
Since \(\omega\) is a primitive \(n\)-th root of unity, we have
\[\sum_{j=0}^{n-1}\omega^{j(\ell-a-b)}=
\begin{cases}
n, & a+b\equiv \ell \pmod n,\\
0, & a+b\not\equiv \ell \pmod n.
\end{cases}\]
It follows that
\[\widehat{\boldsymbol{x}\odot\boldsymbol{y}}_{\ell}=\frac1n
\sum_{\substack{a,b\in[n]\\ a+b\equiv \ell \pmod n}}\widehat{\boldsymbol{x}}_a\widehat{\boldsymbol{y}}_b.\]
\end{proof}

Using the DFT, we explicitly construct cyclic codes whose maximum weights attain the upper and lower bounds in Theorem \ref{thm3.1}.
\begin{exmp}\label{Exmp3.1}
Let \(n=q^2-1\). Then \(h_q(q^2-1)=1\) (see \cite{Cameron-Ellis-Raynaud} or Lemma \ref{lem1.1}). We construct a cyclic code attaining the lower bound in Theorem \ref{thm3.1}.
Let \(\omega\in\mathbb F_{q^2}\) be a primitive \(n\)-th root of unity. Since $\operatorname{ord}_n(q)=2$, the size of the $q$-cyclotomic coset $C_1=\{1,q\}$ is therefore $2$. Define
\[f(X)=(X-\omega)(X-\omega^q)\in\mathbb F_q[X].\]
Then \(f(X)\) is an irreducible quadratic divisor of \(X^n-1\). Let \(\mathcal C=\ker f(\tau)\subseteq\mathbb F_q^n.\) Equivalently, in \(R_n=\mathbb F_q[X]/(X^n-1)\),
\[\mathcal C=\langle g(X)\rangle,\qquad g(X)=\frac{X^n-1}{f(X)}.\]
Since \(\deg g=n-2\), we have
\(\dim \mathcal C=2.\)

We now compute its weights. For every nonzero \(\boldsymbol c\in \mathcal C\), one has $\operatorname{supp}(\widehat{\boldsymbol c})=\{1,q\}.$ Thus we may write
\[\widehat{\boldsymbol c}_1=a,\qquad \widehat{\boldsymbol c}_q=a^q,\qquad a\in\mathbb F_{q^2}^\ast,\]
and all other DFT coordinates are zero. By the IDFT,
\[c_j=\frac1n\left(a\omega^{-j}+a^q\omega^{-qj}\right),
\qquad j=0,1,\dots,n-1.\]
Hence \(c_j=0\) if and only if
\(\omega^{(q-1)j}=-a^{q-1}.\)
Now \(\omega^{q-1}\) has order \(q+1\), and the map
\(j\mapsto \omega^{(q-1)j}\)
from \(\mathbb Z/n\mathbb Z\) onto \(\langle\omega^{q-1}\rangle\) has kernel of size \(q-1\). Moreover,
\((-a^{q-1})^{q+1}=1,\)
so the above equation has exactly \(q-1\) solutions. Therefore every nonzero codeword of \(\mathcal C\) has exactly \(q-1\) zero coordinates, and hence
\[\operatorname{wt}(\boldsymbol c)=n-(q-1)=q^2-q.\]
Thus \(\mathcal C\) is a constant-weight code and
\(d(\mathcal C)=D(\mathcal C)=q^2-q.\)
Consequently,
\(\mathcal C:[q^2-1, 2, q^2-q, q^2-q]_q.\)

Finally, for \(k=2\) and \(n=q^2-1\), the lower bound in Theorem \ref{thm3.1} gives
\[\left\lceil \frac{n(q-1)q^{k-1}}{q^k-1}\right\rceil =\left\lceil\frac{(q^2-1)(q-1)q}{q^2-1}\right\rceil=q^2-q.\]
Hence this code attains the lower bound.
\end{exmp}

We next give examples that achieve the upper bound in Theorem \ref{thm3.1}. 
\begin{exmp}\label{Exmp3.2}
Set \(q=2\), and let \(n>1\) be an odd integer; then we have \(h_2(n)>0\) (see \cite{Cameron-Ellis-Raynaud} or Lemma \ref{lem1.1}).
Construct the binary even-weight cyclic code
\[\mathcal C=\big\{\boldsymbol c=(c_0,\dots,c_{n-1})\in \mathbb F_2^n : c_0+c_1+\cdots+c_{n-1}=0\big\}.\]
Since $n$ is odd, the unique full-weight vector in \(\mathbb F_2^n\) is \((1,1,\dots,1)\), which is not contained in $\mathcal C$. This yields 
\(D(\mathcal C)\le n-1.\)
On the other hand, the vector \((0,1,1,\dots,1)\) lies in the code $\mathcal C$, so
\(D(\mathcal C)=n-1.\)

In addition, it is straightforward to verify that $\mathcal C$ has dimension \(n-1\) and minimum distance \(d(\mathcal C)=2\). Therefore the parameters of $\mathcal C$ are given by
\(\mathcal C:[n,n-1,2,n-1]_2.\)
It is worth noting that this code is an MDS code.
\end{exmp}

\begin{exmp}\label{Exmp3.3}
Let \(q=2^a\), where \(a\ge 1\), and set $n=q+1.$ Then \(\operatorname{ord}_{q+1}(q)=2\). Let \(\omega\in\mathbb F_{q^2}\) be a primitive \(n\)-th root of unity. Since \(q\equiv -1\pmod{q+1}\), we have $\omega^q=\omega^{-1}.$ Consider the \(q\)-cyclotomic coset \(C_1=\{1,q\}.\) Define
\[f(X)=(X-\omega)(X-\omega^q)\in\mathbb F_q[X].\]
Then \(f(X)\) is an irreducible quadratic factor of \(X^{q+1}-1\). Let
\(\mathcal C=\ker f(\tau)\subseteq\mathbb F_q^{q+1}.\)
Equivalently, in \(R_{q+1}=\mathbb F_q[X]/(X^{q+1}-1),\) we have
\[\mathcal C=\langle g(X)\rangle,\qquad
g(X)=\frac{X^{q+1}-1}{f(X)}.\]
Since \(\deg g=q-1\), it follows that
\(\dim \mathcal{C}=(q+1)-(q-1)=2.\)

For any nonzero codeword \(\boldsymbol c\in \mathcal C\), one has \(\operatorname{supp}(\widehat{\boldsymbol c})=\{1,q\}.\) Hence we may write
\[\widehat{\boldsymbol c}_1=a,\qquad
\widehat{\boldsymbol c}_q=a^q,\qquad a\in\mathbb F_{q^2}^\ast,\]
with all other DFT coordinates equal to zero. By the IDFT,
\[c_j=\frac1n\left(a\omega^{-j}+a^q\omega^{-qj}\right).\]
Since in \(\mathbb F_q\) we have \(n=q+1\equiv 1\), and since \(\omega^q=\omega^{-1}\), this becomes
\(c_j=a\omega^{-j}+a^q\omega^j.\)
As the characteristic is \(2\), the condition \(c_j=0\) is equivalent to
\(a\omega^{-j}=a^q\omega^j,\) or \(\omega^{2j}=\frac{a}{a^q}.\)

Now \(\left(\frac{a}{a^q}\right)^{q+1}=1,\) so \(\frac{a}{a^q}\in\langle\omega\rangle\). Since \(|\langle\omega\rangle|=q+1\) is odd, the map
\(j\mapsto \omega^{2j}\)
is a bijection on \(\mathbb Z/(q+1)\mathbb Z\). Therefore the equation
\(\omega^{2j}=\frac{a}{a^q}\)
has exactly one solution. Hence every nonzero codeword of \(\mathcal C\) has exactly one zero coordinate, and so
\(\operatorname{wt}(\boldsymbol c)=q.\)
Thus \(\mathcal C\) is a constant-weight code and
\(d(\mathcal C)=D(\mathcal C)=q.\)

Consequently,
\(\mathcal C:[q+1, 2, q, q]_q.\) It is worth noting that this code is an MDS code. In particular, \(\mathcal C\) contains no full-weight codeword, so \(h_q(q+1)\ge 1\). On the other hand,
\(h_q(q+1)\le \lfloor \log_q(q+1)\rfloor=1.\)
Therefore \(h_q(q+1)=1.\) This result has not appeared in existing literature.
\end{exmp}

Finally, we give an example of odd characteristic.
\begin{exmp}\label{Exmp3.4}
Let \(q=3\) and the code length \(n=13\). We have $h_3(13)=2$ (see \cite{Cameron-Ellis-Raynaud} or Lemma \ref{lem1.1}). Construct the ring  \(R_{13} = \mathbb{F}_3[X]/(X^{13}-1),\) on which we define the cyclic code \(\mathcal C = \langle g(X) \rangle\), with generator polynomial
\[g(X)=(X^3 + 2X + 2)(X^3 + X^2 + X + 2)(X + 2) \in \mathbb{F}_3[X].\]
As \(\deg g = 7\), the dimension of $\mathcal{C}$ is 
\(\dim \mathcal{C} = 13 - 7 = 6.\)
Furthermore, enumerating all \(3^6 = 729\) codewords of $\mathcal C$ gives its weight enumerator polynomial:
\[W_{\mathcal C}(y) = 1 + 156y^6 + 494y^9 + 78y^{12}.\]
Therefore, the ternary code $\mathcal C$ has parameters \([13,6,6,12]_3\).
\end{exmp}

In Theorem \ref{thm3.1}, we derived bounds for \(D(\mathcal C)\). For the minimum distance \(d(\mathcal C)\), we have a general bound
\[2\le d(\mathcal C)\le \left\lfloor \frac{n(q - 1)q^{k-1}}{q^k - 1} \right\rfloor,\]
where \(\mathcal C\ne \mathbb{F}_q^n\). For the case \(h_q(n)=0\), we can obtain sharper bounds.

\begin{thm}\label{thm3.2}
If \(h_q(n)=0\), let \(\mathcal C\subset \mathbb{F}_q^n\) be a nonzero cyclic code such that \(\mathcal C\neq \mathbb{F}_q^n\). Let $k$ denote the dimension of $\mathcal C$. If \(k=1\), then \(d(\mathcal C)=n\); if \(k\ge 2\), we have
\[2\le d(\mathcal C)\le \left\lfloor \frac{n(q - 1)}{q} \right\rfloor.\]
\end{thm}

\begin{proof}
Since \(h_q(n)=0\), every nonzero cyclic code \(\mathcal C\subset \mathbb{F}_q^n\) contains codewords of full weight. If \(k=1\), all codewords of $\mathcal C$ are full-weight, so \(d(\mathcal C)=n\). 

When \(k\ge 2\), pick a full-weight codeword \(\boldsymbol c\in \mathcal C\). Then
\[\lambda \boldsymbol c,\quad \lambda \in \mathbb{F}_q^*\]
gives at least \(q-1\) nonzero codewords of weight $n$. Removing these \(q-1\) full-weight codewords, we obtain \(q^k - q\) remaining nonzero codewords whose total weight is at most
\[\sum_{0\ne \boldsymbol x\in \mathcal C} \mathrm{wt}(\boldsymbol x) -n(q - 1)=n(q - 1)q^{k-1} - n(q - 1)= n(q - 1)\big(q^{k-1} - 1\big).\]
Hence the average weight of the leftover codewords is
\[\frac{n(q - 1)\big(q^{k-1} - 1\big)}{q^k - q} = \frac{n(q - 1)}{q}.\]
This implies there exists at least one nonzero codeword such that
\(\mathrm{wt}(\boldsymbol x) \le \frac{n(q - 1)}{q},\)
and consequently
\[d(\mathcal C) \le \left\lfloor \frac{n(q - 1)}{q} \right\rfloor.\]
\end{proof}

We now give several examples where \(d(\mathcal C)\) attains the upper bound.
\begin{exmp}\label{exmp3.5}
For any \(q \ge 3\), take \(n=3\), then \(h_q(3)=0\) (see \cite{Aaronson-Groenland-Johnston} or Lemma \ref{lem1.2}). Let
\[\mathcal C = \{(x_0,x_1,x_2) \in \mathbb{F}_q^3 : x_0 + x_1 + x_2 = 0\}.\]
This is a cyclic code of length $3$, whose parameters are
\(\mathcal C = [3,2,2,3]_q.\) It is worth noting that this code is an MDS code. In this case, the upper bound is
\[\left\lfloor \frac{3(q - 1)}{q} \right\rfloor = \left\lfloor 3 - \frac{3}{q} \right\rfloor = 2=d(\mathcal C).\]
\end{exmp}

\begin{exmp}\label{exmp3.6}
Set \(q=2\) and \(n=2^m\) with \(m\ge 2\). Then \(h_2(2^m)=0\) (see \cite{Cameron-Ellis-Raynaud} or Lemma \ref{lem1.1}).
Take the cyclic code \(\mathcal C = \langle (X + 1)^{n-2} \rangle\) inside \(R_n = \mathbb{F}_2[X]/(X^n - 1).\) Since \(X^n - 1 = (X + 1)^n,\) the dimension of this code satisfies
\(\dim \mathcal{C} = n - (n - 2) = 2.\)

Moreover,
\[(X + 1)^{n-2} = 1 + X^2 + X^4 + \dots + X^{n-2},\]
so the weight of its corresponding codeword is \(\frac n2\). Meanwhile,
\[X(X + 1)^{n-2} = X + X^3 + \dots + X^{n-1}\]
also has weight \(\frac n2\), and their sum is the all-one vector of weight $n$. Hence \(D(\mathcal C)=n\). The upper bound here equals
\[\left\lfloor \frac{n(2 - 1)}{2} \right\rfloor = \frac{n}{2}.\]
Thus
\[d(\mathcal C) = \frac{n}{2} = \left\lfloor \frac{n(q - 1)}{q} \right\rfloor.\]
Therefore, the parameters of $\mathcal C$ are given by
\(\mathcal C = \left[n,2,\frac n2,n\right]_2.\)
\end{exmp}

\begin{exmp}\label{exmp3.7}
Let \(q\) be an odd prime power, and set $n=q+1.$ Then \(h_q(q+1)=0\) (see \cite{Huang} or Lemma \ref{lem1.3}). Let \(\omega\in\mathbb F_{q^2}^\ast\) be a primitive \(n\)-th root of unity. Since \(q\equiv -1\pmod{q+1}\), we have $\omega^q=\omega^{-1}.$ Consider the \(q\)-cyclotomic coset \(C_1=\{1,q\}.\) Define the polynomial
\[f(X)=(X-\omega)(X-\omega^q)\in\mathbb F_q[X].\]
This polynomial is an irreducible quadratic factor of \(X^{q+1}-1\) over \(\mathbb F_q\). We define the cyclic code
\(\mathcal C=\ker f(\tau)\subseteq \mathbb F_q^{q+1}.\)
Equivalently, in the quotient ring \(R_{q+1}=\mathbb F_q[X]/\langle X^{q+1}-1\rangle,\) the code \(\mathcal C\) is the principal ideal generated by
\[g(X)=\frac{X^{q+1}-1}{f(X)}.\]
Since \(\deg f=2\), we have \(\deg g=q-1\). Therefore
\(\dim \mathcal{C}=(q+1)-(q-1)=2.\)

We now compute its weights. Every nonzero codeword \(\boldsymbol c\in \mathcal C\) satisfies \(\operatorname{supp}(\widehat{\boldsymbol c})=\{1,q\}.\) Take any nonzero element \(a\in\mathbb F_{q^2}\), and define the DFT coordinates by
\[\widehat{\boldsymbol c}_1=a,\qquad \widehat{\boldsymbol c}_q=a^q,\qquad \widehat{\boldsymbol c}_k=0\quad\text{for all }k\notin\{1,q\}.\]
By the IDFT, we obtain
\[c_j=\frac{1}{n}\left(a\omega^{-j}+a^q\omega^{-qj}\right), \qquad j=0,1,\dots,q.\]
Since in \(\mathbb F_q\) we have \(n=q+1\equiv 1\), and since \(\omega^q=\omega^{-1}\), this simplifies to \(c_j=a\omega^{-j}+a^q\omega^j.\) The equality \(c_j=0\) is equivalent to \(a\omega^{-j}+a^q\omega^j=0,\) or equivalently, \(a\omega^{-j}=-a^q\omega^j.\) After rearranging, we get \(\omega^{2j}=-\frac{a}{a^q}.\)

Observe that \(\left(-\frac{a}{a^q}\right)^{q+1}=1.\) Hence \(-\frac{a}{a^q}\in \langle \omega\rangle.\) Now \(\langle\omega\rangle\) has order \(q+1\). Since \(q\) is odd, \(q+1\) is even, and therefore \(\gcd(2,q+1)=2.\) It follows that the map \(j\mapsto \omega^{2j}\) from \(\mathbb Z/(q+1)\mathbb Z\) to \(\langle\omega\rangle\) has kernel of size \(2\), and its image is the index-two subgroup \(\langle\omega^2\rangle\subseteq \langle\omega\rangle.\)

Consequently, the equation \(\omega^{2j}=-\frac{a}{a^q}\) has either no solution or exactly two solutions in \(\mathbb Z/(q+1)\mathbb Z\). Thus every nonzero codeword \(\boldsymbol c\in \mathcal C\) has either no zero coordinate or exactly two zero coordinates. Therefore the possible nonzero weights of \(\mathcal C\) are \(q+1\) and \(q-1.\)

Both weights actually occur. Indeed, the map \(\mathbb F_{q^2}^\ast\longrightarrow \langle\omega\rangle, a\longmapsto \frac{a}{a^q}\) is surjective: its kernel is \(\mathbb F_q^\ast\), so its image has size
\[\frac{q^2-1}{q-1}=q+1.\]
Hence \(-\frac{a}{a^q}\) can be chosen either inside \(\langle\omega^2\rangle\) or outside \(\langle\omega^2\rangle\). In the first case, the corresponding codeword has exactly two zero coordinates and hence weight \(q-1\). In the second case, it has no zero coordinate and hence weight \(q+1\).

We conclude that \(d(\mathcal C)=q-1, D(\mathcal C)=q+1.\) Therefore the parameters of this cyclic code are
\(\mathcal C:[q+1,2,q-1,q+1]_q.\)
It is worth noting that this code is an NMDS code. Finally,
\[\left\lfloor \frac{n(q-1)}{q}\right\rfloor=\left\lfloor \frac{(q+1)(q-1)}{q}\right\rfloor=\left\lfloor q-\frac{1}{q}\right\rfloor=q-1=d(\mathcal C).\]
This shows that the upper bound for \(d(\mathcal C)\) in the case \(h_q(n)=0\) is sharp.
\end{exmp}

Finally, we present an example where \(d(\mathcal C)\) achieves the lower bound $2$ stated in Theorem \ref{thm3.2}.
\begin{exmp}\label{exmp3.8}
Let \(q\) be an odd prime power, and set \(n=q+1\). Then \(h_q(q+1)=0\) (see \cite{Huang} or Lemma \ref{lem1.3}). Consider the cyclic code \(\mathcal C=\langle X-1\rangle\subseteq R_n=\mathbb F_q[X]/\langle X^n-1\rangle .\)
Equivalently,
\[\mathcal C=\left\{(c_0,\ldots,c_{n-1})\in\mathbb F_q^n:\sum_{i=0}^{n-1}c_i=0\right\}.\]
Since \(\deg(X-1)=1\), we have \(\dim \mathcal{C}=n-1=q.\)

We now compute its minimum distance. The codeword corresponding to \(X-1\) is \((-1,1,0,\ldots,0),\) which has weight \(2\). Therefore \(d(\mathcal C)=2.\) Moreover, \(\mathcal C\) contains a full-weight codeword. Since \(n=q+1\) is even, the vector \((1,-1,1,-1,\ldots,1,-1)\) belongs to \(\mathcal C\) and has weight \(n\). Hence \(D(\mathcal C)=n=q+1.\) 

Thus, the parameters of this cyclic code are \(\mathcal C:[q+1, q, 2, q+1]_q.\) It is worth noting that this code is an MDS code.
\end{exmp}

It is worth mentioning that Examples \ref{Exmp3.1} and \ref{Exmp3.2} illustrate that \(d(\mathcal C)\) can achieve the general upper and lower bounds when \(h_q(n)>0\).

\section{Cyclic Codes Containing No Full-Weight Codewords}
In Theorem 2.1, we established a connection between cyclically covering subspaces of \(\mathbb F_q^n\) and the existence of full-weight codewords in cyclic codes over \(\mathbb F_q\). This motivates us to study cyclic codes containing no full-weight codeword. For brevity, let
\[\mathcal N(n,q)=\{\mathcal C\subseteq \mathbb F_q^n : \mathcal C \text{ is a cyclic code and } \mathcal C \text{ contains no full-weight codeword}\}.\]
Equivalently, \(\mathcal C\in\mathcal N(n,q)\) if and only if \(D(\mathcal C)<n\). We use \(\mathcal C_0\) to denote the zero cyclic code. Then \(\mathcal C_0\in\mathcal N(n,q)\), and by Theorem 2.1 we have
\[h_q(n)=0 \Longleftrightarrow |\mathcal N(n,q)|=1,\]
and
\[h_q(n)\ge 1 \Longleftrightarrow |\mathcal N(n,q)|\ge 2.\]

In the general case, it is difficult to compute \(|\mathcal N(n,q)|\), the number of cyclic codes containing no full-weight codeword. However, for $q=2$, we have the following result.

\begin{thm}\label{thm3.3}
Set \(q=2\). Let \(n=2^s m\), with $m$ odd and $s$ a non-negative integer. Then
$$|\mathcal{N}(n,2)|=(2^s+1)^r,$$
where \(r=\sum_{\substack{d\mid m\\d>1}}\frac{\varphi(d)}{\mathrm{ord}_d(2)}.\)
\end{thm}

\begin{proof}
Let $R_n=\mathbb F_2[X]/(X^n-1).$ As \(\mathbb{F}_2^* = \{1\}\), the unique full-weight codeword contained in \(R_n\) is the all-one vector
\(\boldsymbol 1=(1,1,\ldots,1).\)
The vector \(\boldsymbol 1\) corresponds to the polynomial \(S_n(X)=1+X+\cdots+X^{n-1}.\) Hence, for a binary cyclic code \(\mathcal C\), we have
\[D(\mathcal C)=n \Longleftrightarrow \boldsymbol 1\in \mathcal C.\]
Equivalently,
\[\mathcal C\in \mathcal N(n,2) \Longleftrightarrow \boldsymbol 1\notin \mathcal C.\]
Since the characteristic is \(2\), we have \(X^n-1=X^{2^s m}-1=(X^m-1)^{2^s}.\)
Because \(m\) is odd, \(X^m-1\) is square-free over \(\mathbb F_2\). Factor it as \(X^m-1=(X+1)f_1(X)f_2(X)\cdots f_r(X),\) where \(f_1,\ldots,f_r\) are the distinct monic irreducible factors different from \(X+1\). Thus
\[X^n-1=(X+1)^{2^s}f_1(X)^{2^s}\cdots f_r(X)^{2^s}.\]

Every cyclic code of length \(n\) over \(\mathbb F_2\) can be uniquely written as
\[\mathcal C_{\mathbf e}=\left\langle (X+1)^{e_0}f_1(X)^{e_1}\cdots f_r(X)^{e_r}\right\rangle\subseteq R_n,\]
where \(0\le e_i\le 2^s, 0\le i\le r.\) Since \(S_n(X)=1+X+\cdots+X^{n-1}=\frac{X^n-1}{X+1},\) we obtain
\[S_n(X)=(X+1)^{2^s-1}f_1(X)^{2^s}\cdots f_r(X)^{2^s}.\]
Set
\[g_{\mathbf e}(X)=(X+1)^{e_0}f_1(X)^{e_1}\cdots f_r(X)^{e_r}.\]
Because \(g_{\mathbf e}(X)\mid X^n-1\), we have
\[S_n(X)\in \langle g_{\mathbf e}(X)\rangle\Longleftrightarrow g_{\mathbf e}(X)\mid S_n(X).\]

From the explicit factorization of \(S_n(X)\), the divisibility \(g_{\mathbf e}(X)\mid S_n(X)\) fails if and only if \(e_0=2^s.\) Therefore,
\[\mathcal C_{\mathbf e}\in \mathcal N(n,2)\Longleftrightarrow e_0=2^s.\]
Each \(e_i\) (for \(1\le i\le r\)) has \(2^s+1\) choices. Hence there are \((2^s+1)^r\) choices for \((e_1,\ldots,e_r)\).
Thus
\[|\mathcal N(n,2)|=(2^s+1)^r.\]

Finally, we express \(r\) explicitly in arithmetic terms. Since \(m\) is odd,
\[X^m-1=\prod_{d\mid m}\Phi_d(X).\]
The factor \(\Phi_1(X)=X-1=X+1\) accounts for the factor \(X+1\). For each divisor \(d>1\) of \(m\), the irreducible factors of \(\Phi_d(X)\) over \(\mathbb F_2\) all have degree \(\operatorname{ord}_d(2)\). Hence \(\Phi_d(X)\) has \(\frac{\varphi(d)}{\operatorname{ord}_d(2)}\) irreducible factors over \(\mathbb F_2\). Summing over all divisors \(d>1\) of \(m\), we obtain
\[r=\sum_{\substack{d\mid m\\ d>1}}\frac{\varphi(d)}{\operatorname{ord}_d(2)}.\]
This proves the theorem.
\end{proof}

\begin{coro}\label{coro2.3}
\(h_2(n)=0\) if and only if \(n=2^s\) for some non-negative integer $s$.
\end{coro}

\begin{proof}
Theorem \ref{thm3.3} implies \(h_2(n) = 0\) if and only if \(|\mathcal{N}(n,2)| = 1\).
By Theorem~\ref{thm3.3}, we have
\[|\mathcal N(n,2)|=(2^s+1)^r,\]
where
\(r=\sum_{\substack{d\mid m\\ d>1}}\frac{\varphi(d)}{\operatorname{ord}_d(2)}.\)
Since \(2^s+1>1\), we have
\[|\mathcal N(n,2)|=1 \Longleftrightarrow (2^s+1)^r=1 \Longleftrightarrow r=0.\]

Now \(r=0\) if and only if \(m=1\). Indeed, if \(m=1\), then there is no divisor \(d\mid m\) with \(d>1\), so \(r=0\). Conversely, if \(m>1\), then \(m\) has at least one divisor \(d>1\). Since \(m\) is odd, we have \(\gcd(d,2)=1\), so \(\operatorname{ord}_d(2)\) is well-defined and positive. Hence
\(\frac{\varphi(d)}{\operatorname{ord}_d(2)}>0,\)
which implies \(r>0\). Thus \(r=0\) if and only if \(m=1\). Therefore,
\[h_2(n)=0\iff m=1\iff n=2^s.\]
This proves the corollary.
\end{proof}

\begin{coro}\label{coro2.4}
If \(2\) is a primitive root modulo \(n\), then  
\[|\mathcal{N}(n,2)|=2^{d(n)-1},\]  
where \(d(n)\) denotes the number of positive divisors of \(n\). In particular, if \(n\) is prime, then  
\[|\mathcal{N}(n,2)|=2.\]
\end{coro}

\begin{proof}
Since \(2\) is a primitive root modulo \(n\), we have \(\gcd(2,n)=1\), hence \(n\) is odd. Write \(n=2^s m,\) with \(m\) odd. Then \(s=0\) and \(m=n\).
By Theorem \ref{thm3.3}, we have \(|\mathcal{N}(n,2)|=(2^s+1)^r,\) where \(r=\sum_{\substack{d\mid m\\ d>1}}\frac{\varphi(d)}{\operatorname{ord}_d(2)}.\)
Since \(s=0\) and \(m=n\), this becomes  
\[|\mathcal{N}(n,2)|=2^r,\]  
with \(r=\sum_{\substack{d\mid n\\ d>1}}\frac{\varphi(d)}{\operatorname{ord}_d(2)}.\)

Now, for every divisor \(d\mid n\) with \(d>1\), the reduction map \((\mathbb Z/n\mathbb Z)^* \rightarrow (\mathbb Z/d\mathbb Z)^*\) is surjective. Since \(2\) is a primitive root modulo \(n\), its image modulo \(d\) generates \((\mathbb Z/d\mathbb Z)^*\); hence \(2\) is also a primitive root modulo \(d\), and therefore  
\(\operatorname{ord}_d(2)=\varphi(d).\)
Thus  
\[r=\sum_{\substack{d\mid n\\ d>1}}1=d(n)-1.\]  
Consequently, \(|\mathcal{N}(n,2)|=2^r=2^{d(n)-1}.\) If \(n\) is prime, then \(d(n)=2\), so \(|\mathcal{N}(n,2)|=2.\) This completes the proof.
\end{proof}

In 2026, Li and Yuan \cite{Li-Yuan1} proved that if $n>3$ is prime and $\mathrm{ord}_n(3)$ is odd, then $h_3(n)\ge 1$. We can further derive the following results.

\begin{thm}\label{thm3.4}
Let $q=3$, let $n>3$ be a prime number, and suppose that $\mathrm{ord}_n(3)$ is odd. Then
$$|\mathcal{N}(n,3)| \ge 3^{\frac{r}{2}},$$
where $r = \frac{n-1}{\mathrm{ord}_n(3)}.$
\end{thm}

\begin{proof}
Let
\(f=\operatorname{ord}_n(3).\)
Since \(n>3\) is prime, every nonzero \(3\)-cyclotomic coset modulo \(n\) has size exactly \(f\). Hence the total number of nonzero \(3\)-cyclotomic cosets is
\(r=\frac{n-1}{f}.\)
Because \(f\) is odd, the map
\(T\longmapsto -T\)
pairs up all the nonzero \(3\)-cyclotomic cosets. Consequently, the \(r\) nonzero cosets can be partitioned into \(\frac r2\) unordered pairs:
\[\{T_1,-T_1\},\{T_2,-T_2\},\dots,\{T_{\frac r2},-T_{\frac r2}\}.\]
For each index \(i=1,\dots,\frac r2\), we have three choices:
\[\text{choose neither }T_i\text{ nor }-T_i,\qquad\text{choose only }T_i,\qquad\text{choose only }-T_i.\]
Through this construction we obtain \(3^{\frac r2}\) distinct subsets \(S\) of the family of nonzero \(3\)-cyclotomic cosets, and each subset satisfies \(S\cap(-S)=\varnothing.\)
For any such subset \(S\), let \(\mathcal C_S\) be the cyclic code over \(\mathbb{F}_3\) determined by \(S\). 

We claim that every code of the form \(\mathcal C_S\) contains no full-weight codeword.
Assume to the contrary that \(\mathcal C_S\) contains a full-weight codeword
\(\boldsymbol c=(c_0,c_1,\dots,c_{n-1}).\)
When \(q=3\), every nonzero coordinate of \(\boldsymbol c\) lies in
\(\mathbb{F}_3^*=\{\pm1\}.\)
Hence \(c_i^2=1\) for all indices \(i\), and the inner product satisfies
\[(\boldsymbol c,\boldsymbol c)=\sum_{i=0}^{n-1}c_i^2=n.\]
Since \(n>3\) is prime, \(3\nmid n\); thus in the finite field \(\mathbb{F}_3\),
\((\boldsymbol c,\boldsymbol c)=n\neq0.\)
On the other hand, because \(S\cap(-S)=\varnothing\), we have
\[(\boldsymbol c,\boldsymbol c)=\frac1n\sum_{k=0}^{n-1}\widehat{\boldsymbol c}_k\widehat{\boldsymbol c}_{-k}=0.\]
This contradicts \((\boldsymbol c,\boldsymbol c)\neq0\). Hence \(\mathcal C_S\) contains no full-weight codeword. Therefore \(\mathcal C_S\in \mathcal{N}(n,3).\)
There are \(3^{\frac r2}\) distinct choices for \(S\), each giving a distinct cyclic code containing no full-weight codeword, so we obtain the lower bound \(|\mathcal{N}(n,3)|\ge 3^{\frac r2}.\)
\end{proof}

\begin{coro}\label{coro3.3}
Let $q=3$, let $n>3$ be a prime number. If \(\operatorname{ord}_n(3)=\dfrac{n-1}{2}\) and \(\dfrac{n-1}{2}\) is odd, then 
\[|\mathcal{N}(n,3)|=3.\]
\end{coro}

\begin{proof}
Let $f=\operatorname{ord}_n(3)=\frac{n-1}{2}$; then $r=2$. By Theorem \ref{thm3.4}, we have \(|\mathcal N(n,3)| \ge 3\). It suffices to prove that there are at most two nonzero cyclic codes in \(\mathbb F_3^n\) containing no full-weight codeword. Since \(f\) is odd, there are exactly two nonzero \(3\)-cyclotomic cosets modulo \(n\), denoted by
\[C_1=\langle 3\rangle,\qquad C_2=-C_1.\]
Together with \(C_0=\{0\}\), these are all the \(3\)-cyclotomic cosets modulo \(n\).
Let \(W_0,W_1,W_2\) be the corresponding minimal cyclic codes. Then
\[\mathbb F_3^n=W_0\oplus W_1\oplus W_2,\]
where \(W_0=\langle (1,1,\ldots,1)\rangle .\) We now determine which cyclic codes can contain no full-weight codeword. Any cyclic code containing \(W_0\) contains the full-weight codeword
\((1,1,\ldots,1),\) and hence contains a full-weight codeword. Thus the only possible cyclic codes containing no full-weight codeword are
\[W_1,~W_2,~W_1\oplus W_2.\]

We next show that \(W_1\oplus W_2\) contains a full-weight codeword. Since \(W_1\oplus W_2\) corresponds to the union of the two nonzero cyclotomic cosets, it is exactly the cyclic code whose DFT support does not contain \(0\). Equivalently,
\[W_1\oplus W_2=\left\{\boldsymbol x=(x_0,\ldots,x_{n-1})\in\mathbb F_3^n: \sum_{i=0}^{n-1}x_i=0 \right\}.\]
We construct a full-weight vector in this hyperplane. Choose an integer \(r\) such that
\[2r\equiv n \pmod 3.\]
This is possible since \(2\) is invertible in \(\mathbb F_3\). Let \(\boldsymbol x\in\mathbb F_3^n\) have \(r\) coordinates equal to \(1\) and the remaining \(n-r\) coordinates equal to \(-1\). Then \(\boldsymbol x\) has full weight and
\[\sum_{i=0}^{n-1}x_i=r-(n-r)=2r-n \equiv 0\pmod 3.\]
Hence \(x\in W_1\oplus W_2\). Thus \(W_1\oplus W_2\) contains a full-weight codeword.

Finally, since \(|\mathcal N(n,3)|\ge 3\), both \(W_1\) and \(W_2\) must contain no full-weight codeword, which gives \(|\mathcal N(n,3)| = 3\).
\end{proof}

\begin{exmp}\label{Exmp3.9}
When \(n=11\), we have \(\operatorname{ord}_{11}(3)=5\). Therefore, by Corollary \ref{coro3.3}, we obtain \(|\mathcal{N}(11,3)|=3\).
\end{exmp}

The lower bound in Theorem \ref{thm3.4} is not necessarily sharp. When \(n=13\), enumerating all cyclic codes gives \(|\mathcal{N}(13,3)|=11\).
This indicates that the lower bound in Theorem \ref{thm3.4} is not tight, and we propose the following conjecture.

\begin{conj}\label{conj3.1}
Let \(q=3\), let \(n>3\) be a prime, and assume that \(\mathrm{ord}_n(3)\) is odd. Then
\[|\mathcal{N}(n,3)| = \sum_{j=0}^{\frac r2}\binom{r}{j},\]
where \(r = \frac{n-1}{\mathrm{ord}_n(3)}.\)
\end{conj}

Let \(f=\operatorname{ord}_n(3)\), and assume that \(f\) is odd.
Since \(n>3\) is prime, the total number of nonzero $3$-cyclotomic cosets is
\(r=\frac{n-1}{f}.\)
Denote these cosets by \(C_1,C_2,\dots,C_r.\) Let the DFT support of the cyclic code $\mathcal C$ satisfy
\(\operatorname{supp}(\widehat{\mathcal{C}})=\bigcup_{s=1}^t C_{i_s}.\)
We state the following conjecture.
\begin{conj}\label{conj3.2}
Let the DFT support of the cyclic code $\mathcal C$ satisfy
\(\operatorname{supp}(\widehat{\mathcal C})=\bigcup_{s=1}^t C_{i_s}.\) Then $\mathcal C\in\mathcal{N}(n,3)\) if and only if \(0\notin \operatorname{supp}(\widehat{\mathcal C})\) and \(t\le \dfrac{r}{2}\).
\end{conj}

If Conjecture \ref{conj3.2} is true, then Conjecture \ref{conj3.1} is also true. Assume Conjecture~\ref{conj3.2} holds. Then a cyclic code belongs to \(\mathcal N(n,3)\) precisely when its DFT support is the union of \(j\) nonzero \(3\)-cyclotomic cosets with \(0\le j\le \frac r2\). Since there are \(\binom rj\) ways to choose such \(j\) cosets among the \(r\) nonzero cosets, we obtain
\[|\mathcal N(n,3)|=\sum_{j=0}^{\frac r2}\binom rj,\]
which is exactly Conjecture~\ref{conj3.1}.

In Theorem \ref{thm3.4}, we constructed cyclic codes over \(\mathbb F_3^n\) containing no full-weight codeword using the condition \(S\cap(-S)=\emptyset\). Furthermore, we have the following result.

\begin{thm}\label{thm3.5}
Let \(n\geq 2\) be an integer satisfying \(\gcd(n,3)=1\), and let \(S\) be a union of \(3\)-cyclotomic cosets. Let
\[\mathcal C_S=\{\boldsymbol x\in\mathbb F_3^n:\operatorname{supp}(\widehat{x})\subseteq S\}.\]
Assume that for every nonempty union \(T\subseteq S\) of \(3\)-cyclotomic cosets, at least one of the following two conditions holds:

(1) \(T\cap(-T)=\emptyset\);

(2) there exists some \(\ell\in [n]\setminus\{0\}\) such that the equation
\[a+b\equiv \ell \pmod n,\qquad a,b\in T,\]
has a unique solution up to interchanging \(a\) and \(b\).

Then \(\mathcal C_S\) contains no full-weight codeword. In particular, if \(S\ne\emptyset\), then \(\mathcal C_S\in \mathcal N(n,3)\) is nonzero.
\end{thm}

\begin{proof}
Suppose, for contradiction, that \(\mathcal C_S\) contains a full-weight codeword \(\boldsymbol x=(x_0,\dots,x_{n-1})\). Since \(\boldsymbol x\in\mathbb F_3^n\) is full-weight, we have
\(x_i\in\mathbb F_3^*=\{\pm1\}\) for every \(i\). Hence
\[x_i^2=1,\qquad i=0,\dots,n-1,\]
or equivalently,
\[\boldsymbol x\odot \boldsymbol x=(1,1,\dots,1),\]
where \(\boldsymbol x\odot \boldsymbol x\) denotes the coordinatewise square of \(\boldsymbol x\). Let \(T=\operatorname{supp}(\widehat{\boldsymbol x}).\)
Since \(\boldsymbol x\in \mathcal C_S\), we have \(T\subseteq S\). Also \(T\neq\emptyset\), since \(\boldsymbol x\neq0\).
Taking the DFT of \(\boldsymbol x\odot \boldsymbol x=(1,\dots,1)\), we obtain
\[\widehat{(\boldsymbol x\odot \boldsymbol x)}_0=n,\qquad \widehat{(\boldsymbol x\odot \boldsymbol x)}_\ell=0\quad(\ell\neq0).\]
On the other hand, by Lemma \ref{lem3-7}, we get
\[\widehat{(\boldsymbol x\odot \boldsymbol x)}_\ell=\frac1n\sum_{a+b=\ell}\widehat{\boldsymbol x}_a\widehat{\boldsymbol x}_b.\]
Therefore
$$\sum_{a+b=\ell} \widehat{\boldsymbol x}_a \widehat{\boldsymbol x}_b =
\begin{cases}
n^2, & \ell = 0, \\
0, & \ell \neq 0.
\end{cases}$$

First suppose that \(T\cap(-T)=\emptyset.\) Then there are no \(a,b\in T\) satisfying \(a+b=0\). Hence
\[\sum_{a+b=0}\widehat{\boldsymbol x}_a\widehat{\boldsymbol x}_b=0.\]
Thus \(\widehat{(\boldsymbol x\odot \boldsymbol x)}_0=0,\) contradicting \(\widehat{(\boldsymbol x\odot \boldsymbol x)}_0=n\). Since \(\gcd(n,3)=1\), the image of \(n\) in \(\mathbb F_3\) is nonzero.

Next suppose that there exists some \(\ell\neq0\) such that \(\ell\) has a unique unordered representation
\[\ell=a+b,\qquad a,b\in T.\]
Then the convolution sum
\[\sum_{u+v=\ell}\widehat{\boldsymbol x}_u\widehat{\boldsymbol x}_v\]
has only the contribution coming from this unique unordered pair.

If \(a=b\), then
\[\sum_{u+v=\ell}\widehat{\boldsymbol x}_u\widehat{\boldsymbol x}_v=\widehat{\boldsymbol x}_a^2\neq0,\]
because \(a\in T\) implies \(\widehat{\boldsymbol x}_a\neq0\).

If \(a\neq b\), then the two ordered pairs \((a,b)\) and \((b,a)\) occur, so
\[\sum_{u+v=\ell}\widehat{\boldsymbol x}_u\widehat{\boldsymbol x}_v=\widehat{\boldsymbol x}_a\widehat{\boldsymbol x}_b+\widehat{\boldsymbol x}_b\widehat{\boldsymbol x}_a=2\widehat{\boldsymbol x}_a\widehat{\boldsymbol x}_b\neq0.\]
Indeed, the computation takes place in a field of characteristic \(3\), so \(2\neq0\), and \(\widehat{\boldsymbol x}_a,\widehat{\boldsymbol x}_b\neq0\).

In either case,
\[\sum_{u+v=\ell}\widehat{\boldsymbol x}_u\widehat{\boldsymbol x}_v\neq0,\]
which contradicts
\[\widehat{(\boldsymbol x\odot \boldsymbol x)}_\ell=\frac1n\sum_{u+v=\ell}\widehat{\boldsymbol x}_u\widehat{\boldsymbol x}_v=0.\]
Thus no full-weight codeword can exist in \(\mathcal C_S\). Therefore \(\mathcal C_S\) contains no full-weight codeword whenever \(S\ne\emptyset\).
\end{proof}
We refer to the first condition in Theorem \ref{thm3.5} as the \textit{orthogonality obstruction}, and the second condition as the \textit{unique representation obstruction}.

\begin{exmp}\label{exmp3.9}
In Theorem \ref{thm3.5}, we use the orthogonality obstruction to give a lower bound for \(|\mathcal N(n,q)|\), and prove that this bound is sharp when \(\operatorname{ord}_n(3)=\dfrac{n-1}{2}\) and \(\dfrac{n-1}{2}\) is odd, as for \(n=11\). However, this is not the case for \(n=13\). In the following, we improve this bound for \(n=13\) by using the unique representation obstruction.

The nonzero $3$-cyclotomic cosets modulo $13$ are
\[C_1 = \{1, 3, 9\},\quad C_2 = \{2, 5, 6\},\quad -C_1 = \{4, 10, 12\},\quad -C_2 = \{7, 8, 11\}.\]
The orthogonality obstruction gives \(3^2= 9\) cyclic codes containing no full-weight codeword in total, including the zero code. However, the unique representation obstruction can additionally rule out two more:
\[ C_1 \cup (-C_1), \quad C_2 \cup (-C_2). \]
Specifically, for
\[ T_1 = C_1 \cup (-C_1) = \{1, 3, 4, 9, 10, 12\}, \]
we have
\[ 1 = 4 + 10 \pmod{13}, \]
and this is the unique unordered representation in \(T_1\). Hence \(\mathcal C_{T_1}\) contains no full-weight codeword.

For
\[ T_2 = C_2 \cup (-C_2) = \{2, 5, 6, 7, 8, 11\}, \]
we have
\[ 2 = 7 + 8 \pmod{13}, \]
and this is the unique unordered representation. Hence this \(\mathcal C_{T_2}\) also contains no full-weight codeword.

So the unique representation obstruction raises the lower bound from $9$ to $11$. In fact, by enumerating all cases, we obtain the exact value \(|\mathcal N(13, 3)| = 11.\)
\end{exmp}

\begin{exmp}\label{exmp3.10} Using Theorem \ref{thm3.5}, we can obtain the lower bound \(|\mathcal N(8,3)|\ge 3\). A direct enumeration shows that this lower bound is sharp, namely \(|\mathcal N(8,3)|=3\). The nonzero $3$-cyclotomic cosets modulo $8$ are  
\[C_1 = \{1, 3\},\quad C_2 = \{2, 6\},\quad C_3 = \{4\},\quad C_4 = \{5, 7\}.\]  
We have \(-C_1 = C_4, -C_4 = C_1,\) while \(-C_2 = C_2, -C_3 = C_3.\)
Let the corresponding minimal cyclic codes be denoted by  
\[W_0, W_1, W_2, W_3, W_4.\]  
Then  
\[\mathbb{F}_3^8 = W_0 \oplus W_1 \oplus W_2 \oplus W_3 \oplus W_4.\]  
Since  
\[ C_1 \cap (-C_1) = C_1 \cap C_4 = \emptyset, \]  
we have that \(W_1\) contains no full-weight codeword. Similarly,  
\[ C_4 \cap (-C_4) = C_4 \cap C_1 = \emptyset, \]  
so \(W_4\) also contains no full-weight codeword.
\end{exmp}

\begin{exmp}\label{exmp3.11}
Using Theorem \ref{thm3.5}, we can also obtain a sharp lower bound for
\(|\mathcal N(16,3)|\). 
The nonzero \(3\)-cyclotomic cosets modulo \(16\) are
\[C_1=\{1,3,9,11\},\quad C_2=\{2,6\},\quad C_3=\{4,12\},\]
\[C_4=\{5,7,13,15\},\quad C_5=\{8\},\quad C_6=\{10,14\}.\]
We have \(-C_1=C_4, -C_4=C_1,\) and \(-C_2=C_6, -C_6=C_2,\) while \(-C_3=C_3, -C_5=C_5.\)
Let the corresponding minimal cyclic codes be denoted by
\[W_0,~W_1,~W_2,~W_3,~W_4,~W_5,~W_6.\]
Then
\[\mathbb F_3^{16}=W_0\oplus W_1\oplus W_2\oplus W_3\oplus W_4\oplus W_5\oplus W_6.\]

By Theorem \ref{thm3.5}, any cyclic code whose DFT support \(S\) satisfies
\(S\cap(-S)=\emptyset\)
contains no full-weight codeword. Hence we may choose at most one coset from each of the two opposite pairs
\[\{C_1,C_4\},\qquad \{C_2,C_6\},\]
and we cannot choose the self-negative cosets \(C_3\) and \(C_5\). Therefore the following eight nonzero cyclic codes contain no full-weight codeword:
\[W_1,\quad W_4,\quad W_2,\quad W_6,\quad W_1\oplus W_2,\quad W_1\oplus W_6,\quad W_4\oplus W_2,\quad W_4\oplus W_6.\]
Together with the zero code, this gives
\(|\mathcal N(16,3)|\ge 9.\)
A direct enumeration of all cyclic codes of length \(16\) over \(\mathbb F_3\) shows that no other cyclic code contains no full-weight codeword. Hence
\(|\mathcal N(16,3)|=9.\)
\end{exmp}

When \(n>3\) is prime and \(\operatorname{ord}_n(3)\) is even, the orthogonality obstruction fails, whereas the unique representation obstruction can still apply.

\begin{exmp}\label{exmp3.12}
Let \(n=41\). Then $\operatorname{ord}_{41}(3)=8,$ and in fact \(3^4\equiv -1 \pmod {41}.\)
Hence \(-1\in \langle 3\rangle\), so every nonzero \(3\)-cyclotomic coset modulo \(41\) is invariant under negation. Therefore the orthogonality obstruction cannot be applied to any nonzero minimal cyclic code.

The nonzero \(3\)-cyclotomic cosets modulo \(41\) are
\[C_1=\{1,3,9,27,40,38,32,14\},\]
\[C_2=\{2,6,18,13,39,35,23,28\},\]
\[C_3=\{4,12,36,26,37,29,5,15\},\]
\[C_4=\{7,21,22,25,34,20,19,16\},\]
\[C_5=\{8,24,31,11,33,17,10,30\}.\]
Since \(3^4\equiv -1\pmod {41}\), we have \(-C_i=C_i, 1\le i\le 5.\)
Thus \(C_i\cap(-C_i)=C_i\neq\emptyset,\) so the orthogonality obstruction fails for each \(C_i\).
However, the unique representation obstruction applies. Indeed, in \(C_1\), we have
\[4=1+3\pmod {41},\]
and this is the unique unordered representation of \(4\) as a sum of two elements of \(C_1\).
Hence the minimal cyclic code \(W_1\) corresponding to \(C_1\) contains no full-weight codeword by Theorem \ref{thm3.5}.

Similarly, we have the following unique unordered representations:
\[8=2+6\pmod {41}\quad \text{in }C_2,\qquad 1=37+5\pmod {41}\quad \text{in }C_3,\]
\[2=21+22\pmod {41}\quad \text{in }C_4,\qquad 2=33+10\pmod {41}\quad \text{in }C_5.\]
Therefore the corresponding minimal cyclic codes
\[W_1,\ W_2,\ W_3,\ W_4,\ W_5\]
all contain no full-weight codeword. Hence
\(|\mathcal N(41,3)|\ge 6.\)
This example shows that the unique representation obstruction can produce cyclic codes
containing no full-weight codeword even in cases where the orthogonality obstruction is completely ineffective.
\end{exmp}
From Example \ref{exmp3.12}, we obtain \(h_3(41)>0\). In fact, this is not an isolated phenomenon, and we may further generalize this result. We first generalize Lemma \ref{lem3-7} accordingly.

\begin{lem}\label{lem4-1}
Assume that \(\gcd(n,q)=1\). Let \(\boldsymbol x^{\odot t}\) denote the \(t\)-fold
coordinatewise product of \(\boldsymbol x\) with itself. Then for every integer \(t\ge 2\)
and every \(\ell\in[n]\), we have
\[\widehat{\boldsymbol x^{\odot t}}_\ell=\frac{1}{n^{t-1}}\sum_{a_1+\cdots+a_t\equiv \ell \pmod n} \widehat{\boldsymbol x}_{a_1}\widehat{\boldsymbol x}_{a_2}\cdots \widehat{\boldsymbol x}_{a_t}.\]
\end{lem}

\begin{proof}
We prove the formula by induction on \(t\). For $t=2$, this is exactly the result given in Lemma \ref{lem3-7}. Assume that the formula holds for some \(t\ge 2\). Since \(\boldsymbol x^{\odot(t+1)}=\boldsymbol x^{\odot t}\odot \boldsymbol x,\) the twofold convolution formula gives
\[\widehat{\boldsymbol x^{\odot(t+1)}}_\ell=\frac{1}{n}\sum_{b+c\equiv \ell \pmod n}\widehat{\boldsymbol x^{\odot t}}_b\widehat{\boldsymbol x}_c.\]
By the induction hypothesis,
\[\widehat{\boldsymbol x^{\odot t}}_b=\frac{1}{n^{t-1}}\sum_{a_1+\cdots+a_t\equiv b \pmod n}\widehat{\boldsymbol x}_{a_1}\cdots\widehat{\boldsymbol x}_{a_t}.\]
Substituting this into the previous identity, we get
\[\widehat{\boldsymbol x^{\odot(t+1)}}_\ell=\frac{1}{n^t}\sum_{b+c\equiv \ell} \sum_{a_1+\cdots+a_t\equiv b}\widehat{\boldsymbol x}_{a_1}\cdots\widehat{\boldsymbol x}_{a_t}\widehat{\boldsymbol x}_c.\]
The conditions
\[a_1+\cdots+a_t\equiv b,\qquad b+c\equiv \ell\]
are equivalent to
\[a_1+\cdots+a_t+c\equiv \ell \pmod n.\]
Renaming \(c\) as \(a_{t+1}\), we obtain
\[\widehat{\boldsymbol x^{\odot(t+1)}}_\ell=\frac{1}{n^t}\sum_{a_1+\cdots+a_{t+1}\equiv \ell \pmod n}\widehat{\boldsymbol x}_{a_1}\cdots\widehat{\boldsymbol x}_{a_{t+1}}.\]
This proves the induction step. Hence the formula holds for all \(t\ge2\).
\end{proof}

\begin{thm}\label{thm4-4}
Let \(q\ge 3\) be an odd prime, let \(n\ge 2\) be an integer satisfying \(\gcd(n,q)=1\), and let \(S\) be a union of \(q\)-cyclotomic cosets modulo \(n\). Let
\[\mathcal C_S=\{\boldsymbol x\in\mathbb F_q^n : \operatorname{supp}(\widehat{x})\subseteq S\}.\]
For a subset \(T\subseteq \mathbb Z/n\mathbb Z\), write
\[(q-1)T=\{a_1+\cdots+a_{q-1}:a_i\in T\}.\]
Assume that for every nonempty union \(T\subseteq S\) of \(q\)-cyclotomic cosets, at least one of the following two conditions holds:

(1) \(0\notin (q-1)T\);

(2) there exists some \(\ell\in [n]\setminus\{0\}\) such that the congruence
\[a_1+\cdots+a_{q-1}\equiv \ell \pmod n,\qquad a_i\in T,\]
has a unique solution up to permutation of \(a_1,\ldots,a_{q-1}\).

Then \(\mathcal C_S\) contains no full-weight codeword. In particular, if \(S\ne\emptyset\), then \(\mathcal C_S\in\mathcal N(n,q)\) is nonzero.
\end{thm}

\begin{proof}
Suppose, for contradiction, that \(\mathcal C_S\) contains a full-weight codeword
\(x=(x_0,\ldots,x_{n-1})\). Since \(x_i\in\mathbb F_q^*\) for every \(i\), we have
\[x_i^{q-1}=1,\qquad i=0,\ldots,n-1.\]
Equivalently,
\[\boldsymbol x^{\odot(q-1)}=(1,1,\ldots,1),\]
where \(\boldsymbol x^{\odot(q-1)}\) denotes the \((q-1)\)-fold coordinatewise product of \(\boldsymbol x\) with itself.

Let \(T=\operatorname{supp}(\widehat{\boldsymbol x}).\) Since \(\boldsymbol x\in\mathcal C_S\), we have \(T\subseteq S\). Moreover, \(T\) is a nonempty union of \(q\)-cyclotomic cosets. Taking the DFT of \(\boldsymbol x^{\odot(q-1)}=(1,1,\ldots,1),\) we obtain
\[\widehat{\boldsymbol x^{\odot(q-1)}}_0=n,\qquad\widehat{\boldsymbol x^{\odot(q-1)}}_\ell=0\quad(\ell\ne0).\]
Here \(n\ne0\) in \(\mathbb F_q\), since \(\gcd(n,q)=1\). By Lemma \ref{lem4-1}, for every \(\ell\in[n]\) we have
\[\widehat{\boldsymbol x^{\odot(q-1)}}_\ell=\frac{1}{n^{q-2}}\sum_{a_1+\cdots+a_{q-1}\equiv \ell}\widehat{\boldsymbol x}_{a_1}\widehat{\boldsymbol x}_{a_2}\cdots\widehat{\boldsymbol x}_{a_{q-1}}.\]
Since \(\widehat{\boldsymbol x}_a\ne0\) precisely for \(a\in T\), the above sum only receives nonzero contributions from \((q-1)\)-tuples in \(T\).

First suppose that \(0\notin(q-1)T.\) Then there are no \(a_1,\ldots,a_{q-1}\in T\) satisfying
\[a_1+\cdots+a_{q-1}\equiv0\pmod n.\]
Hence the convolution sum at \(0\) is zero, and so \(\widehat{\boldsymbol x^{\odot(q-1)}}_0=0,\) contradicting \(\widehat{\boldsymbol x^{\odot(q-1)}}_0=n\ne0.\)

Next suppose that there exists some \(\ell\ne0\) such that
\[a_1+\cdots+a_{q-1}\equiv \ell\pmod n,\qquad a_i\in T,\]
has a unique solution up to permutation. Write this unique unordered solution as
\[\underbrace{u_1,\ldots,u_1}_{m_1\ \mathrm{times}},\underbrace{u_2,\ldots,u_2}_{m_2\ \mathrm{times}},\ldots,\underbrace{u_s,\ldots,u_s}_{m_s\ \mathrm{times}},\]
where \(u_1,\ldots,u_s\in T\) are distinct and \(m_1+\cdots+m_s=q-1.\) Then the corresponding ordered solutions contribute
\[\frac{(q-1)!}{m_1!m_2!\cdots m_s!}\widehat{\boldsymbol x}_{u_1}^{m_1}\widehat{\boldsymbol x}_{u_2}^{m_2}\cdots\widehat{\boldsymbol x}_{u_s}^{m_s}\]
to the convolution sum at \(\ell\). Since \(q\) is prime and all integers \(m_i\) and \(q-1\) are strictly smaller than \(q\), the multinomial coefficient
\[\frac{(q-1)!}{m_1!m_2!\cdots m_s!}\]
is nonzero in \(\mathbb F_q\). Also \(\widehat{\boldsymbol x}_{u_i}\ne0\) for every \(i\), because \(u_i\in T=\operatorname{supp}(\widehat{\boldsymbol x})\). Therefore
\[\sum_{a_1+\cdots+a_{q-1}\equiv \ell}\widehat{\boldsymbol x}_{a_1}\cdots\widehat{\boldsymbol x}_{a_{q-1}}\ne0.\]
It follows that \(\widehat{\boldsymbol x^{\odot(q-1)}}_\ell\ne0,\) contradicting the fact that \(\widehat{\boldsymbol x^{\odot(q-1)}}_\ell=0, \ell\ne0.\)

Hence \(\mathcal C_S\) contains no full-weight codeword. Finally, if \(S\ne\emptyset\), then \(\mathcal C_S\) is a nonzero cyclic code, since \(S\) is a nonempty union of \(q\)-cyclotomic cosets. Therefore \(\mathcal C_S\in\mathcal N(n,q).\)
\end{proof}

\begin{thm}\label{thm3.6}
Let \(q\ge 3\) be an odd prime, and let \(m\ge 4\) be an integer. Set \(n=\frac{q^m+1}{2}.\) Then \(h_q(n)>0.\)
\end{thm}

\begin{proof}
Write \(q=2r+1\), where \(r=\frac{q-1}{2}\). Since \(2n=q^m+1,\) we have \(q^m\equiv -1\pmod n\), and hence \(q^{2m}\equiv1\pmod n\). Also \(q\nmid n\), since \(2n=q^m+1\equiv 1\pmod q.\)

We first prove that \(\operatorname{ord}_n(q)=2m\). For \(1\le d<m\), we have
\[q^d-1<n,\qquad q^d+1<n,\]
and hence \(q^d\not\equiv \pm1\pmod n\). Also \(q^m\equiv -1\not\equiv1\pmod n\). If \(m<d<2m\), write \(d=m+s\) with \(0<s<m\). If \(q^d\equiv1\pmod n\), then
\(q^{m+s}\equiv1\pmod n.\)
Since \(q^m\equiv -1\pmod n\), this gives \(q^s\equiv -1\pmod n\), contradicting the preceding observation. Therefore no positive exponent smaller than \(2m\) gives \(1\), and so \(\operatorname{ord}_n(q)=2m.\) Thus the \(q\)-cyclotomic coset of \(1\) modulo \(n\) is
\[C_1=\{1,q,\ldots,q^{m-1},-1,-q,\ldots,-q^{m-1}\}.\]
Let \(W_1\) be the minimal cyclic code corresponding to \(C_1\). We prove that \(W_1\) contains no full-weight codeword.

Put \(\ell=r(1+q).\) Since \(m\ge4\), we have \(0<\ell<n\). We first show that \(\ell\) has a unique unordered representation as a sum of \(q-1=2r\) elements of \(C_1\), namely
\[\ell=\underbrace{1+\cdots+1}_{r\ \mathrm{times}}+\underbrace{q+\cdots+q}_{r\ \mathrm{times}}.\]

Suppose
\[a_1+\cdots+a_{2r}\equiv \ell \pmod n,\qquad a_i\in C_1.\]
Choose representatives
\[a_i=\varepsilon_i q^{b_i},\qquad \varepsilon_i\in\{\pm1\},\quad 0\le b_i\le m-1,\]
and set \(S=a_1+\cdots+a_{2r}.\) Then
\[|S|\le 2rq^{m-1}=(q-1)q^{m-1}<q^m<2n.\]
Since \(S\equiv\ell\pmod n\), we may write \(S=\ell+kn\) for some \(k\in\mathbb Z\). The bound \(|S|<2n\) leaves only \(k\in\{-2,-1,0,1\}\). Moreover,
\[2n-\ell-(q-1)q^{m-1}=q^{m-1}+1-\frac{q^2-1}{2}>0,\]
because \(q\ge3\) and \(m\ge4\). Hence
\[|S|\le (q-1)q^{m-1}<2n-\ell,\]
so \(S=\ell-2n\) is impossible. Therefore the congruence \(S\equiv\ell\pmod n\) implies
\[S=\ell,\qquad S=\ell+n,\qquad\text{or}\qquad S=\ell-n.\]

We shall use the following elementary observation. A signed \(q\)-adic expression means an expression
\[M=\sum_i c_iq^i,\qquad c_i\in\mathbb Z,\]
and its coefficient budget is \(\sum_i|c_i|\). Reducing modulo \(q\) forces \(c_0\equiv M\pmod q\). Once \(c_0\) is fixed, the number \((M-c_0)/q\) has a signed \(q\)-adic expression with budget decreased by \(|c_0|\).

We discuss the three cases. First suppose \(S=\ell=r+rq.\) Write
\[S=\sum_{i=0}^{m-1}c_iq^i,\qquad \sum_i|c_i|\le 2r.\]
Modulo \(q\), we have \(c_0\equiv r\pmod q\). Since \(|c_0|\le 2r=q-1\), either
\[c_0=r\qquad\text{or}\qquad c_0=r-q=-(r+1).\]
If \(c_0=-(r+1)\), then after subtracting \(c_0\) and dividing by \(q\), the remaining part represents \(r+1\) with total absolute coefficient at most
\[2r-(r+1)=r-1,\]
which is impossible modulo \(q\). Hence \(c_0=r\). Then after subtracting \(r\) and dividing by \(q\), the remaining part represents \(r\) with total absolute coefficient at most \(r\). Therefore necessarily \(c_1=r\), and all other \(c_i\)'s are zero. Since the total absolute coefficient is already \(2r\), no cancellation is possible. Hence the representation is uniquely
\(\ell=r\cdot 1+r\cdot q.\)

Next suppose \(S=\ell+n.\) Since
\[n=1+r(1+q+\cdots+q^{m-1}),\]
we have
\[\ell+n=q^2\bigl((r+1)+rq+rq^2+\cdots+rq^{m-3}\bigr).\]
Let
\[A=(r+1)+rq+rq^2+\cdots+rq^{m-3}.\]
Since \(S=q^2A\), any signed \(q\)-adic expression for \(S\) with budget at most \(2r=q-1\) must have its constant and \(q\)-coefficients equal to zero. Hence \(A\) would admit a signed \(q\)-adic expression with total absolute coefficient at most \(2r\).

We show that this is impossible. Reducing such an expression for \(A\) modulo \(q\), the constant coefficient must be congruent to \(r+1\pmod q\). Since its absolute value is at most \(2r=q-1\), it must be either
\[r+1\qquad\text{or}\qquad (r+1)-q=-r.\]
If it is \(r+1\), then after subtracting it and dividing by \(q\), the remaining quotient is congruent to \(r\pmod q\), but the remaining coefficient budget is at most
\[2r-(r+1)=r-1,\]
which is impossible. If it is \(-r\), then after subtracting it and dividing by \(q\), the quotient is again of the same form, but with coefficient budget at most \(r\). Reducing once more modulo \(q\), the constant coefficient is forced to be \(-r\), which uses up the whole remaining budget, while the remaining number is still nonzero. This contradiction excludes \(S=\ell+n\).

Finally suppose \(S=\ell-n.\) Equivalently, \(n-\ell=-S\) would admit a signed \(q\)-adic expression with total absolute coefficient at most \(2r\). But
\[n-\ell=1+rq^2+rq^3+\cdots+rq^{m-1}.\]
Reducing modulo \(q\), the constant coefficient must be either
\[1 \qquad\text{or}\qquad 1-q=-2r.\]
The second possibility already uses all \(2r\) coefficients but leaves a nonzero remainder, impossible. Hence the constant coefficient is \(1\). After subtracting \(1\), the remaining number is divisible by \(q^2\). Since the remaining budget is at most \(2r-1<q\), the coefficient of \(q\) must be zero. Therefore, after dividing by \(q^2\), we obtain a signed expression for
\[r(1+q+\cdots+q^{m-3})\]
with total absolute coefficient at most \(2r-1\).

Since \(m\ge4\), the above sum has at least two terms. Reducing modulo \(q\), the constant coefficient must be either
\[r \qquad\text{or}\qquad r-q=-(r+1).\]
If it is \(r\), then after subtracting it and dividing by \(q\), the remaining quotient is nonzero and still congruent to \(r\pmod q\), while the remaining budget is at most \(r-1\), impossible. If it is \(-(r+1)\), then for \(r=1\) this already exceeds the available budget; for \(r\ge2\), after subtracting it and dividing by \(q\), the remaining budget is at most \(r-2\), while the quotient is congruent to \(r+1\pmod q\), again impossible. Thus \(S=\ell-n\) is also excluded.

Therefore \(\ell\) has the unique unordered representation
\[\ell=\underbrace{1+\cdots+1}_{r}+\underbrace{q+\cdots+q}_{r}\]
as a sum of \(q-1=2r\) elements of \(C_1\).

Since \(C_1\) is a single \(q\)-cyclotomic coset, the only nonempty union of \(q\)-cyclotomic cosets contained in \(C_1\) is \(C_1\) itself. Hence the above unique representation verifies condition \((2)\) of Theorem \ref{thm4-4} for \(S=C_1\). Therefore \(W_1\) is a nonzero cyclic code containing no full-weight codeword. By Theorem \ref{lem1}, we conclude that \(h_q(n)>0.\)
\end{proof}

\begin{rem}\label{rem4-1}
For example, when \(q=3\) and \(m=4\), we have \(n=\frac{3^4+1}{2}=41,\) which is a prime number. In addition, for every fixed odd prime \(q\), the family in Theorem \ref{thm3.6} contains infinitely many composite integers. Indeed, let \(m\ge 4\) be an integer which is not a power of \(2\). Write
\[m=2^s u,\qquad u>1 \text{ odd}.\]
Then \(q^m+1=\bigl(q^{2^s}\bigr)^u+1\) is divisible by \(q^{2^s}+1\). Hence
\[\frac{q^m+1}{2}=\frac{q^{2^s}+1}{2}\left(\bigl(q^{2^s}\bigr)^{u-1}-\bigl(q^{2^s}\bigr)^{u-2}+\cdots-q^{2^s}+1\right).\]
Both factors on the right-hand side are integers greater than \(1\). Therefore \(\frac{q^m+1}{2}\) is composite whenever \(m\ge 4\) is not a power of \(2\).

Consequently, for every fixed odd prime \(q\), Theorem \ref{thm3.6} gives an explicit infinite family of integers \(n=\frac{q^m+1}{2} ~(m\ge4)\) such that \(h_q(n)>0\), and this family contains infinitely many composite values of \(n\).
\end{rem}

\section{Algebraic Consequences of the Cyclic-Code Criterion}
In 2025, Li and Yuan \cite{Li-Yuan} used group algebra theory to show that characterizing \(h_q(n) = 0\) can be reduced to the coprime case \(\gcd(q, n) = 1\). More precisely, they proved that for every positive integer $k$, \(h_q(np^k)=0\) is equivalent to \(h_q(n)=0\). We now present a new proof by virtue of Theorem \ref{lem1}.

For convenience of notation, we say that the polynomial \(\boldsymbol{\alpha}(X)\) is full-weight in \(R_n\) whenever \(\boldsymbol{\alpha}\) is a full-weight codeword in \(\mathbb{F}_q^n\).
\begin{thm}\label{lem2}
Let \( q \) be a power of a prime $p$, and let \( n \) be a positive integer satisfying \( \gcd(p, n) = 1 \). Then for any positive integer \( k \), we have \( h_q(np^k) = 0 \) if and only if \( h_q(n) = 0 \).
\end{thm}

\begin{proof}
By virtue of Theorem \ref{lem1}, it suffices to prove that every nonzero cyclic code over \(R_n\) contains a full-weight codeword if and only if every nonzero cyclic code over \(R_{p^kn}\) contains a full-weight codeword.

First, we have
\[X^{p^kn} - 1 = (X^n)^{p^k}- 1 = (X^n - 1)^{p^k}.\]
Let \(F(X) = X^n - 1\). Since \(\gcd(n,p) = 1\), \(F(X)\) is square-free. Let
\(F(X) = f_1(X)f_2(X)\cdots f_r(X)\)
be the irreducible factorization of \(F(X)\) in \(\mathbb{F}_q[X]\). It follows that
\[X^{p^kn} - 1 = F(X)^{p^k} = f_1(X)^{p^k}\cdots f_r(X)^{p^k}.\]
Hence every cyclic code over \(R_{p^kn}\) is of the form
\[\left\langle \prod_{j=1}^r f_j(X)^{e_j} \right\rangle,\quad 0 \leq e_j \leq p^k,\]
whereas all cyclic codes over \(R_n\) take the shape
\[\left\langle \prod_{j\in J} f_j(X) \right\rangle,\quad J \subseteq \{1,\dots,r\}.\]

On the other hand, we have
\[F(X)^{p^k-1} = (X^n - 1)^{p^k-1} = 1 + X^n + X^{2n} + \dots + X^{(p^k-1)n}.\]
Thus if
\[\boldsymbol{c}(X) = c_0 + c_1 X + \dots + c_{n-1} X^{n-1} \in R_n,\]
then
\[F(X)^{p^k-1}\boldsymbol{c}(X) = \boldsymbol{c}(X) + X^n \boldsymbol{c}(X) + \dots + X^{(p^k-1)n}\boldsymbol{c}(X) \in R_{p^kn}.\]
This corresponds to repeating the length-$n$ vector \(\boldsymbol{c}\) exactly $p^k$ times. Hence, we have
\[\boldsymbol{c}(X) \text{ is full-weight in } R_n \iff F(X)^{p^k-1}\boldsymbol{c}(X) \text{ is full-weight in } R_{p^kn}.\]

Suppose
\(h_q(n) = 0.\)
Equivalently, every nonzero cyclic code over \(R_n\) contains a full-weight codeword. Now take an arbitrary nonzero cyclic code \(\mathcal C = \langle g(X) \rangle\) over \(R_{p^kn}\).
Since \(g(X) \mid F(X)^{p^k}\), we may write
\[g(X) = \prod_{j=1}^r f_j(X)^{e_j},\quad 0 \leq e_j \leq p^k.\]
As \(\mathcal C \neq \{0\}\), we have
\(g(X) \neq F(X)^{p^k},\)
which means that not all \(e_j\) equal $p^k$. Let
\[J = \{j : e_j = p^k\},\]
and define
\[d(X) = \prod_{j\in J} f_j(X).\]
If $J$ is the empty set, we define \(d(X)=1\) by convention.
Since not all \(e_j = p^k\), it follows that
\(d(X) \neq F(X).\)
Then
\(\mathcal D = \langle d(X) \rangle \subseteq R_n\)
is a nonzero cyclic code. By the assumption \(h_q(n) = 0\), $\mathcal D$ admits a full-weight codeword. That is, there exists some
\(\boldsymbol{c}(X) \in \langle d(X) \rangle \subseteq R_n\)
such that \(\boldsymbol{c}(X)\) is full-weight. Now consider
\[\boldsymbol{A}(X) = F(X)^{p^k-1}\boldsymbol{c}(X) \in R_{p^kn}.\]
Then \(\boldsymbol{A}(X)\) is full-weight of length $p^kn$. We now verify that
\(\boldsymbol{A}(X) \in \mathcal C = \langle g(X) \rangle.\)
We examine the exponent of each irreducible factor \(f_j\): 

If \(j \in J\), then \(e_j = p^k\). Since \(\boldsymbol{c}(X) \in \langle d(X) \rangle\), we have \(f_j \mid \boldsymbol{c}(X)\). Thus the exponent of \(f_j\) in \(F(X)^{p^k-1}\boldsymbol{c}(X)\) is at least
\((p^k-1)+1 = p^k = e_j.\)

If \(j \notin J\), then \(e_j \leq p^k-1\). The term \(F(X)^{p^k-1}\) already contains the factor \(f_j^{p^k-1}\), which supplies a sufficiently high exponent.

Therefore
\(g(X) \mid F(X)^{p^k-1}\boldsymbol{c}(X),\)
and hence
\(\boldsymbol{A}(X) \in \langle g(X) \rangle = \mathcal C.\)
Thus $\mathcal C$ contains a full-weight codeword. As $\mathcal C$ was chosen to be an arbitrary nonzero cyclic code over \(R_{p^kn}\), every nonzero cyclic code over \(R_{p^kn}\) contains a full-weight codeword. Consequently,
\(h_q(p^kn) = 0.\)

Now suppose
\(h_q(p^kn) = 0.\)
Equivalently, every nonzero cyclic code over \(R_{p^kn}\) contains a full-weight codeword. Take an arbitrary nonzero cyclic code \(\mathcal D = \langle d(X) \rangle \subseteq R_n.\)
Since
\(F(X) = f_1(X)\cdots f_r(X),\)
we have
\[d(X) = \prod_{j\in J} f_j(X),\]
where \(J \subsetneq \{1,\dots,r\}\). Now consider the cyclic code over \(R_{p^kn}\) defined by
\[\mathcal C = \langle F(X)^{p^k-1}d(X) \rangle.\]
This code is nonzero, because at least one irreducible factor \(f_j\) appears with exponent exactly \(p^k-1\) in \(F(X)^{p^k-1}d(X)\), which is strictly less than $p^k$. By the assumption \(h_q(p^kn) = 0\), $\mathcal C$ admits a full-weight codeword, denoted
\(\boldsymbol{A}(X) \in \mathcal C.\) As
\(\mathcal C = \langle F(X)^{p^k-1}d(X) \rangle,\)
there exists some \(b(X) \in R_{p^kn}\) such that
\[\boldsymbol{A}(X) = F(X)^{p^k-1}d(X)b(X).\]
Reduce \(b(X)\) modulo \(F(X)\) and write
\[b(X) = b_0(X) + F(X)b_1(X).\]
In \(R_{p^kn}\), we have the identity
\(F(X)^{p^k} = 0,\)
which implies
\[F(X)^{p^k-1}d(X)b(X) = F(X)^{p^k-1}d(X)b_0(X).\]
Set
\[\boldsymbol{c}(X) = d(X)b_0(X) \in R_n.\]
Then
\(\boldsymbol{c}(X) \in \mathcal D = \langle d(X) \rangle,\)
and
\(\boldsymbol{A}(X) = F(X)^{p^k-1}\boldsymbol{c}(X).\) We know that \(\boldsymbol{A}(X)\) is full-weight if and only if \(\boldsymbol{c}(X)\) is full-weight.
Thus $\mathcal D$ contains a full-weight codeword. Since $\mathcal D$ was chosen as an arbitrary nonzero cyclic code over \(R_n\), we conclude
\(h_q(n) = 0.\)
\end{proof}

Theorem \ref{lem1} reduces the condition \(h_q(n)=0\) to the problem of whether there exists a full-weight codeword in cyclic codes over \(\mathbb{F}_q^n\). For the case \(h_q(n)\geq k\), we obtain the analogous theorem below.

Let \(\boldsymbol{\alpha}_1,\boldsymbol{\alpha}_2,\dots,\boldsymbol{\alpha}_k\in \mathbb{F}_q^n\) be a set of linearly independent vectors. We can construct a linear subspace of codimension $k$:
\[V_{(\boldsymbol{\alpha}_1,...,\boldsymbol{\alpha}_k)} = V_{\boldsymbol{\alpha}_1}\cap V_{\boldsymbol{\alpha}_2}\cap \dots \cap V_{\boldsymbol{\alpha}_k},\]
where \(V_{\boldsymbol{\alpha}_j} = \{\boldsymbol{x}\in \mathbb{F}_q^n : (\boldsymbol{x},\boldsymbol{\alpha}_j) = 0\}\), $j=1,...,k$. For each \(\boldsymbol{\alpha}_j\), construct its corresponding reciprocal polynomial 
\[\boldsymbol{\alpha}_j^*(X)=\boldsymbol{\alpha}_j(X^{-1}) \pmod{X^n - 1}.\]
Taking the polynomial tuple \(\big(\boldsymbol{\alpha}_1^*(X),\boldsymbol{\alpha}_2^*(X),\dots,\boldsymbol{\alpha}_k^*(X)\big)\) as generators, we define the single-generator quasi-cyclic code:
\[\begin{split}
\mathcal C(\boldsymbol{\alpha}_1,\dots,\boldsymbol{\alpha}_k) &= \Big\{ \big(\boldsymbol{x}(X)\boldsymbol{\alpha}_1^*(X),\boldsymbol{x}(X)\boldsymbol{\alpha}_2^*(X),\,\dots,\,\boldsymbol{x}(X)\boldsymbol{\alpha}_k^*(X)\big) : \boldsymbol{x}(X)\in R_n \Big\}\\
&=\Big\{ \Big(\boldsymbol{x}(X)\cdot\big(\boldsymbol{\alpha}_1^*(X),\boldsymbol{\alpha}_2^*(X),\,\dots,\,\boldsymbol{\alpha}_k^*(X)\big)\Big) : \boldsymbol{x}(X)\in R_n \Big\}\subseteq R_n^k.
\end{split}\]
Every codeword \(\boldsymbol{c}(\boldsymbol{x})\in \mathcal C(\boldsymbol{\alpha}_1,\dots,\boldsymbol{\alpha}_k)\) is a $k$-dimensional polynomial vector, and each component of \(\boldsymbol{c}(\boldsymbol{x})\) belongs to \(R_n\).

For notational convenience, write
\[\boldsymbol{a}(\boldsymbol{x})_j = \theta^{-1}\big(\boldsymbol{x}(X)\boldsymbol{\alpha}_j^*(X)\big)\in \mathbb{F}_q^n.\]
For any codeword \(\boldsymbol{c}(\boldsymbol{x})\in \mathcal C(\boldsymbol{\alpha}_1,\dots,\boldsymbol{\alpha}_k)\), define the {\it $i$-th composite component} of \(\boldsymbol{c}(\boldsymbol{x})\) as
\[(\boldsymbol{c}(\boldsymbol{x}))_i=\big((\boldsymbol{a}(\boldsymbol{x})_1)_i,\,(\boldsymbol{a}(\boldsymbol{x})_2)_i,\,\dots,\,(\boldsymbol{a}(\boldsymbol{x})_k)_i\big)\in \mathbb{F}_q^k.\]
Note that \((\boldsymbol{a}(\boldsymbol{x})_j)_i\) is exactly the coefficient of \(X_i\) in \(\boldsymbol{x}(X)\boldsymbol{\alpha}_j^*(X)\). In addition, for convenience, we also denote \([X^i]\big(\boldsymbol{x}(X)\boldsymbol{\alpha}_j^*(X)\big)\) as the coefficient of \(X_i\) in \(\boldsymbol{x}(X)\boldsymbol{\alpha}_j^*(X)\), i.e.,
\((\boldsymbol{a}(\boldsymbol{x})_j)_i=[X^i]\big(\boldsymbol{x}(X)\boldsymbol{\alpha}_j^*(X)\big).\)

\begin{thm}\label{lem3}
Let \(k\ge 1\), and let \(\boldsymbol{\alpha}_1,\dots,\boldsymbol{\alpha}_k\in\mathbb F_q^n\) be linearly independent. Then \(V_{(\boldsymbol{\alpha}_1,\dots,\boldsymbol{\alpha}_k)}\) is a cyclically covering subspace of \(\mathbb{F}_q^n\) with codimension $k$ if and only if for any \(\boldsymbol{c}(\boldsymbol{x}) \in \mathcal C(\boldsymbol{\alpha}_1,\dots,\boldsymbol{\alpha}_k)\), there exists some \(i \in [n]\) such that \((\boldsymbol{c}(\boldsymbol{x}))_i = \boldsymbol{0}\).
\end{thm}

\begin{proof}
For \(\boldsymbol x\in \mathbb F_q^n\), the \(i\)-th composite component of the associated codeword is
\[(\boldsymbol c(\boldsymbol x))_i=\bigl((\boldsymbol x,\tau^i(\boldsymbol{\alpha}_1)),\ldots,(\boldsymbol x,\tau^i(\boldsymbol{\alpha}_k)\bigr).\]
Thus \((\boldsymbol c(\boldsymbol x))_i=0\) if and only if
\[(\boldsymbol x,\tau^i(\boldsymbol \alpha_1))=\cdots=(\boldsymbol x,\tau^i(\boldsymbol \alpha_k))=0,\]
which is equivalent to \(\boldsymbol x\in \tau^i(V_{(\boldsymbol{\alpha}_1,\ldots,\boldsymbol{\alpha}_k)})\). Therefore every \(\boldsymbol x\in\mathbb F_q^n\) lies in some cyclic shift of \(V_{(\boldsymbol{\alpha}_1,\ldots,\boldsymbol{\alpha}_k)}\) if and only if every codeword in \(\mathcal C(\boldsymbol{\alpha}_1,\ldots,\boldsymbol{\alpha}_k)\) has a zero composite component.
\end{proof}
When \(k=1\), Theorem \ref{lem3} reduces to Theorem \ref{lem1}.
Based on Theorem \ref{lem3}, we give the definition of an admissible subspace in \(R_n\).
\begin{defn}\label{def4.1}
Let \(A = \operatorname{span}_{\mathbb{F}_q}\{\boldsymbol{\alpha}_1^*(X),\boldsymbol{\alpha}_2^*(X),\dots,\boldsymbol{\alpha}_k^*(X)\} \subseteq R_n\). We call the subspace A an admissible subspace of \(R_n\) if for every \(\boldsymbol{c}(\boldsymbol{x}) \in \mathcal C(\boldsymbol{\alpha}_1,\dots,\boldsymbol{\alpha}_k)\), there exists an index \(i \in [n]\) satisfying \((\boldsymbol{c}(\boldsymbol{x}))_i = \boldsymbol{0}\).
\end{defn}
Equivalently, \(A\) is admissible if for every \(\boldsymbol x(X)\in R_n\), there exists \(i\in[n]\) such that \([X^i]\bigl(\boldsymbol x(X)u(X)\bigr)=0\) for all \(u(X)\in A\). Hence admissibility depends only on the subspace \(A\), not on the chosen basis. By Definition \ref{def4.1}, we have
\[h_q(n) = \max\big\{\dim_{\mathbb{F}_q} A : A \subseteq R_n \text{ is admissible}\big\}.\]

In 2019, Aaronson, Groenland and Johnston \cite{Aaronson-Groenland-Johnston} proved that
\(h_q(pn) \leq p h_q(n).\)
By Theorem \ref{lem3}, we now provide an algebraic proof of this inequality.

\begin{thm}\label{thm4.1}
Let $q$ be a power of prime $p$ and $n$ a positive integer. Then we have
\[h_q(pn) \leq p h_q(n).\]
\end{thm}

\begin{proof}
Let \(F(X)=X^n-1\). First, we have
\[F(X)^p = (X^n - 1)^p  = X^{pn} - 1.\]
Thus the quotient ring \(R_{pn}\) can be written as
\[R_{pn} = \mathbb{F}_q[X] \big/ (X^{pn} - 1) = \mathbb{F}_q[X] \big/ (F(X)^p).\]
We construct a chain of subspaces of \(R_{pn}\) using powers of \(F(X)\), which yields the following $F$-adic filtration:
\[R_{pn} = F^0 R_{pn} \supset F^1 R_{pn} \supset F^2 R_{pn} \supset \cdots \supset F^{p-1} R_{pn} \supset F^p R_{pn} = 0,\]
where \(F^t R_{pn} = \big\{ F(X)^t \cdot f(X) : f(X) \in R_{pn} \big\}.\) For every \(t = 0,1,\dots,p-1\), the quotient space \(F^t R_{pn} \big/ F^{t+1} R_{pn}\) is isomorphic to \(R_n = \mathbb{F}_q[X]/(X^n-1)\). We define the linear map that realizes this isomorphism as follows:
\[\varphi_t: F^t R_{pn} \to R_n,\quad F(X)^t f(X) \mapsto f(X) \pmod{X^n-1}.\]
It is straightforward to verify that \(\ker \varphi_t = F^{t+1} R_{pn}\). Therefore, by the First Isomorphism Theorem for modules, we get
\[F^t R_{pn} \big/ F^{t+1} R_{pn} \cong R_n.\]
Let \(A \subseteq R_{pn}\) be an admissible subspace. We shall show that
\[\dim_{\mathbb{F}_q} A \leq ph_q(n).\]
We intersect $A$ with the filtration to obtain a chain of subspaces of $A$:
\[A_t = A \cap F^t R_{pn},\quad t=0,1,\dots,p.\]
It is clear that
\[A = A_0 \supseteq A_1 \supseteq A_2 \supseteq \cdots \supseteq A_p = 0.\]
By the dimension additivity formula for filtered vector spaces, we have
\[\dim_{\mathbb{F}_q} A = \sum_{t=0}^{p-1} \dim_{\mathbb{F}_q}\big( A_t / A_{t+1} \big).\]
We map each quotient \(A_t/A_{t+1}\) into \(R_n\) via the isomorphism constructed above, which gives a subspace \(B_t \subseteq R_n\). We write \(B_t\) explicitly as
\[B_t = \big\{ u(X) \in R_n : \exists\, a(X) \in A_t \text{ such that } a(X) = F(X)^t u(X) + F(X)^{t+1} v(X) \big\}.\]
It is immediate that \(\dim_{\mathbb{F}_q}\big(A_t / A_{t+1}\big) = \dim_{\mathbb{F}_q} B_t.\)
We then prove that every \(B_t\) is an admissible subspace of \(R_n\), so that \(\dim_{\mathbb{F}_q} B_t \le h_q(n)\). Summing these inequalities yields the required upper bound.

By the key identity
\[F(X)^{p-1} = (X^n - 1)^{p-1} = 1 + X^n + X^{2n} + \cdots + X^{(p-1)n},\]
for any \(c(X)\in R_n\), the coefficient of \(X^{\ell+rn}\) in \(F^{p-1}c(X)\) coincides with the coefficient of \(X^\ell\) in \(c(X)\). In other words,
\[[X^{\ell + r n}]\big( F^{p-1} c(X) \big) = [X^\ell]\big( c(X) \big),\]
where \(0\le \ell < n\) and \(0\le r < p\). 

To show that \(B_t\) is admissible, we need to verify that for any \(c(X) \in R_n\), there exists an index \(\ell \in [n]\) such that \([X^\ell]\big(c(X) \cdot u(X)\big) = 0\)
holds for all \(u(X) \in B_t\). Given \(c(X) \in R_n\), construct the polynomial
\[b(X) = F(X)^{p-1-t} \cdot c(X) \in R_{pn}.\]
Since $A$ is an admissible subspace of \(R_{pn}\) and \(A_t \subseteq A\), \(A_t\) is automatically admissible as well. Thus there exists some index \(i \in [pn]\) satisfying \([X^i]\big(b(X) \cdot a(X)\big) = 0\)
for all \(a(X) \in A_t\). We may write $i$ in the form
\[i = \ell + r n,\quad \text{where } 0 \le \ell < n,\ 0 \le r < p.\]
Take an arbitrary \(u(X) \in B_t\). By the definition of \(B_t\), there exists some \(a(X) \in A_t\) such that
\[a(X) = F(X)^t u(X) + F(X)^{t+1} v(X)\]
for some polynomial \(v(X)\). Compute the product \(b(X)a(X)\):
\[\begin{aligned}
b(X) \cdot a(X)
&= F^{p-1-t} c(X) \cdot \big( F^t u(X) + F^{t+1} v(X) \big) \\
&= F^{p-1} c(X) u(X) + F^{p} c(X) v(X) \\
&= F(X)^{p-1} \cdot c(X) u(X).
\end{aligned}\]
Recall the identity
\[[X^{\ell + r n}]\big( F^{p-1} c(X) u(X) \big) = [X^\ell]\big( c(X) u(X) \big).\]
Combined with
\[[X^i](b(X)a(X))=[X^{\ell + r n}](b(X)a(X))=0,\]
we obtain
\([X^\ell]\big( c(X) u(X) \big) = 0.\)
This \(\ell\) is fixed and valid for every \(u(X) \in B_t\), which proves that \(B_t\) is an admissible subspace of \(R_n\). Since \(h_q(n)\) denotes the maximal dimension of any admissible subspace of \(R_n\), we have
\[\dim_{\mathbb{F}_q} B_t \le h_q(n),\quad t = 0,1,\dots,p-1.\]
Using the dimension additivity property of the filtration, we obtain
\[\begin{aligned}
\dim_{\mathbb{F}_q} A&= \sum_{t=0}^{p-1} \dim_{\mathbb{F}_q}\big( A_t / A_{t+1} \big) = \sum_{t=0}^{p-1} \dim_{\mathbb{F}_q} B_t 
\le \sum_{t=0}^{p-1} h_q(n) = p \cdot h_q(n).
\end{aligned}\]
As $A$ is an arbitrary admissible subspace of \(R_{pn}\), the maximal dimension \(h_q(pn)\) of admissible subspaces of \(R_{pn}\) satisfies \(h_q(pn) \le ph_q(n).\)
\end{proof}

\begin{coro}\label{coro4.1}
Let \( q \) be a power of a prime $p$, and let \( n \) be a positive integer satisfying \( \gcd(p, n) = 1 \). Then for any positive integer \( k \), we have \(h_q(n)\leq h_q(np^k)\leq p^k h_q(n)\). In particular, \(h_q(np^k) = 0\) if and only if \(h_q(n) = 0\).
\end{coro}

\end{document}